\documentclass[preprint,12pt,onecolumn]{article}

\usepackage{amsmath}
\usepackage{amssymb}
\usepackage{amsthm}
\usepackage[english]{babel}
\usepackage{tikz}
\usetikzlibrary{intersections, shapes}
\usetikzlibrary{arrows, math}
\usepackage{subcaption}
\usepackage{booktabs,multirow}
\usepackage{bigdelim}
\usepackage[hidelinks]{hyperref}
\usepackage{array}
\usepackage{pgfplots}
\usepackage{graphicx,bm,xcolor}
\theoremstyle{plain} 
\newtheorem{theorem}{Theorem}[section]

\newtheorem{lemma}[theorem]{Lemma}

\newtheorem{remark}[theorem]{Remark}

\theoremstyle{definition} %

\theoremstyle{remark} %

\title{IETI-DP methods for discontinuous Galerkin 
	multi-patch Isogeometric Analysis with T-junctions
	}

\author{Rainer Schneckenleitner\footnote{\texttt{schneckenleitner@numa.uni-linz.ac.at},
		Institute of Computational Mathematics, Johannes
		Kepler University Linz, Austria}\; and Stefan Takacs\footnote{\texttt{stefan.takacs@ricam.oeaw.ac.at},
		Johann Radon Institute Institute for Computational and Applied Mathematics, Austrian Academy of Sciences, Linz, Austria}}
\date{}

\begin{document}
\maketitle
%
%
\selectlanguage{english}
\begin{abstract}
	We study Dual-Primal Isogeometric Tearing and Interconnecting
	(IETI-DP) solvers for non-conforming multi-patch discretizations
	of a generalized Poisson problem. We realize the coupling between
	the patches using a symmetric interior penalty discontinuous Galerkin (SIPG)
	approach. Previously, we have assumed that the interfaces between
	patches always consist of whole edges.
	In this paper, we drop this requirement and allow T-junctions. This
	extension is vital for the consideration of sliding interfaces,
	for example between the rotor and the stator of an electrical
	motor. One critical part for the handling of T-junctions
	in IETI-DP solvers is the choice of the primal degrees of freedom.
	We propose to add all basis functions that are non-zero at any
	of the vertices to the primal space. Since there are several
	such basis functions at any T-junction, we call this concept
	``fat vertices''. For this choice, we show a condition number
	bound that coincides with the bound for the conforming case.
\end{abstract}
%
%
\section{Introduction}
\label{sec:1}
%
%
%

Isogeometric Analysis (IgA),~\cite{HughesCottrellBazilevs:2005}, is
an approach to discretize partial differential equations (PDEs)
that has been designed in order to overcome difficulties related
to meshing of the computational domain. In IgA, the
computational domain is parameterized by geometry functions, which
are commonly represented in terms of B-splines or
non-uniform rational B-splines (NURBS). Such a representation is
also used in state-of-the-art computer aided-design (CAD) software. Usually, one considers multiple patches, each parameterized with its
own geometry function (multi-patch IgA). We consider
the case of non-overlapping patches.

Conforming discretizations require that both the geometry functions and the grids agree on each interface between two patches.
If this is not the case, discontinuous Galerkin (dG) methods,
particularly the
symmetric interior penalty discontinuous Galerkin (SIPG) approach~\cite{Arnold:1982},
are an appropriate option. For its adaptation to IgA, see~\cite{Hofer:2016a,HoferLanger:2017c,HoferLanger:2019b,SchneckenleitnerTakacs:2020}
and others. In these publications, it is assumed that the interfaces between two patches consist (in the two-dimensional case) of whole edges,
which excludes the case of T-junctions between patches. Now,
we include the case of T-junctions, which allows greater flexibility
for the geometry modeling. This is of vital interest for the simulation
of objects with sliding interfaces, like the interface between the rotor
and the stator of an electrical motor, which serves as computational domain in our model problem. The PDE in the model problem is the Poisson equation, which can be motivated as a model for the magnetostatic potential.
In general, sliding interfaces lead to a non-matching
decomposition of the computational domain into patches. In two dimensions, this means that T-junctions between patches occur for most rotational angles.

So far, several approaches have been considered to handle such types of problems, including the more classical locked-step methods, cf.~\cite{PrestonReece:1988}, the moving band technique, cf.~\cite{DavatRen:1985}, the Lagrange multiplier method, cf.~\cite{LaiRodger:1992}, interpolation approaches, cf.~\cite{PerrinCoulomb:1995}. More recently, mortar techniques, cf.~\cite{BuffaMaday:2001,Egger:Harutyunyan:Merkel:Schops:2020}, domain interface methods, cf.~\cite{CafieroLloberas:2016} or discontinuous Galerkin methods, cf.~\cite{AlottoBertoni:2001} have been considered. Also combinations of different approaches have been proposed, see, e.g.,~\cite{KettunenKurz:2014}.

We focus on the SIPG approach, the contribution of this paper is a
fast solver for the linear system obtained from the proposed discretization.
A canonical choice for domains with many non-overlapping patches are domain decomposition (DD) solvers, like the Dual-Primal Finite Element Tearing and Interconnecting (FETI-DP) method, cf.~\cite{FarhatLesoinneLeTallecPiersonRixen:2001a,FarhatRoux:1991a}. These solvers have been adapted to IgA in~\cite{KleissPechsteinJuttlerTomar:2012} and are sometimes referred to as Dual-Primal Isogeometric Tearing and Interconnecting (IETI-DP) solvers. Recently, the authors have
developed an analysis that is also explicit in the spline degree,
see~\cite{SchneckenleitnerTakacs:2019}.

The IETI-DP solvers have been extended to discontinuous Galerkin
discretizations in~\cite{Hofer:2016a,HoferLanger:2016a,HoferLanger:2017c,SchneckenleitnerTakacs:2020}, however
T-junctions have not been covered by the analysis so far. Most components
of the IETI-DP framework can easily be extended to domains with T-junctions, see~\cite{SchneckenleitnerTakacs:2021} for a numerical
study. 
One of the critical questions is the choice of the primal degrees of
freedom. We propose to add all basis functions that are non-zero at
a vertex to the primal space. This yields a number of basis functions
for each T-junction that grows linearly with the spline degree
(fat vertices). If a vertex is the corner of the respective patch, there
is only one single non-zero basis function. This means that
our choice coincides with the standard choice of corner values.
We introduce
a scaled Dirichlet preconditioner for the Schur complement formulation
of the IETI-DP system. We show that the condition number of the preconditioned system is bounded by 
\[
C p \left( 1+\log p + \max_{k=1,\dots,K} \log \frac{H_k}{h_k} \right)^2,
\]
where the constant $C>0$ is independent of the grid sizes $h_k$, the patch sizes $H_k$, the spline degree $p$, the smoothness of the splines within the patches $\Omega^{(k)}$, and coefficient jumps between the patches.
$C$ depends on the geometry functions, the maximum number of
patches that meet on any vertex, the minimal interface length and the quasi-uniformity of the grids
within each patch.

The remainder of the paper is structured as follows. We introduce the
model problem in Section~\ref{sec:2} and its SIPG discretization in Section~\ref{sec:3}. In  Section~\ref{sec:4}, we propose the IETI-DP solver. Numerical examples are presented in the subsequent Section~\ref{sec:5}. The paper is concluded with some final remarks in Section~\ref{sec:6}. The proof of the condition number bound is given in an Appendix.

\section{The model problem}
\label{sec:2}
%
%
%

In this section, we introduce the model problem which we consider
in this paper. We use the same notation as in~\cite{SchneckenleitnerTakacs:2020}. To keep the paper self-contained, we introduce the notation in
the following.

First, we introduce the computational domain. $\Omega \subset \mathbb{R}^2$
is a simply connected and bounded open Lipschitz domain, which
is composed of $K$ non-overlapping, simply connected open
patches $\Omega^{(k)}$, i.e.,
\begin{align*}
	\overline{\Omega} = \bigcup_{k=1}^K \overline{\Omega^{(k)}} \quad \text{and}\quad
	\Omega^{(k)} \cap \Omega^{(\ell)} = \emptyset \quad\text{for all}\quad k \neq \ell,
\end{align*}
where $\overline{T}$ denotes the closure of the set $T$. 
We assume that every patch $\Omega^{(k)}$ is parameterized by a geometry function
\begin{align}
	G_k:\widehat{\Omega}:=(0,1)^2 \rightarrow \Omega^{(k)}:=G_k(\widehat{\Omega}) \subset \mathbb{R}^2, 
\end{align}
that has a continuous extension to the closure of $\widehat\Omega$.
In IgA, the geometry functions $G_k$ are commonly represented in terms of
B-splines or NURBS. For the presented analysis, it is not necessary
to restrict ourselves to these representations, as long as the Jacobian
of $G_k$ and its inverse are uniformly bounded.

We use the common notation for the Lebesgue and Sobolev spaces $L_2(\Omega)$ and $H^s(\Omega)$, $s \in \mathbb{R}$, respectively. Those function spaces are equipped with the standard norms
$\|\cdot\|_{L_2(\Omega)}$ and $\|\cdot\|_{H^s(\Omega)}$
and seminorms
$|\cdot|_{H^s(\Omega)}$. As usual, $H^1_0(\Omega) \subset H^1(\Omega)$ denotes the subspace of functions vanishing on $\partial \Omega$.

The boundary value problem of interest reads as follows.
Find $u \in H^1_0(\Omega)$ such that 
\begin{align}
	\label{continousProb}
	\int_{\Omega}^{} \alpha \nabla u \cdot \nabla v \; \mathrm{d}x = \int_{\Omega}^{} f v \; \mathrm{d}x \qquad\text{for all}\qquad v \in H^1_0(\Omega),
\end{align}
with a given source function $f \in H^{-1}(\Omega)$ and a uniformly positive and bounded diffusion coefficient $\alpha$, which is constant
on each patch, i.e., we have
\[
\alpha(x) = \alpha_k
\quad \mbox{for all}\quad
x\in \Omega^{(k)}
\]
with $\alpha_k>0$ for all $k=1,\ldots,K$. Since, for simplicity,
we represent the Dirichlet boundary conditions in a strong sense,
we assume that the pre-images $\widehat \Gamma_D^{(k)}
:= G_k^{-1}(\partial\Omega\cap\partial\Omega^{(k)})$ of the Dirichlet boundary
$\Gamma_D=\partial \Omega$ consist of whole edges of the parameter
domain $\widehat \Omega$. An alternative, where this restriction would
not be necessary, would be a fully floating IETI-DP discretization.

\section{The discretization using IgA and SIPG}
\label{sec:3}
%
%
%

In this section, we first introduce the patch-local discretization
spaces, then we discuss the overall discretization. For the local
discretization spaces, we restrict ourselves for simplicity to
B-splines. Let $p \in \mathbb{N}:=\{ 1, 2, 3, \dots \}$ be the spline
degree, where we assume for simplicity that the degree is uniform
for all patches. For each patch $k\in \{1,\ldots,K\}$ and each spatial
dimension $\delta \in \{1,2\}$, we introduce a $p$-open knot vector
\[
\Xi^{(k,\delta)}=(\xi_1^{(k,\delta)},\ldots,\xi_{n^{(k,\delta)}+p+1}^{(k,\delta)})
\]
with $\xi_1^{(k,\delta)}=\dots=\xi_{p+1}^{(k,\delta)} = 0$ and $\xi_{n^{(k,\delta)}}^{(k,\delta)} = \dots = \xi_{n^{(k,\delta)}+p+1}^{(k,\delta)}= 1$,
where each inner knot might be repeated up to $p$ times. Depending on the
$p$-open knot vector, we introduce a B-spline basis
$(B[p,\Xi^{(k,\delta)},i])_{i=1}^{n^{(k,\delta)}}$ via the Cox-de~Boor formula, cf.~\cite[Eq.~(2.1) and (2.2)]{Cottrell:Hughes:Bazilevs}.
The collection $(B[p,\Xi^{(k,\delta)},i])_{i=1}^{n^{(k,\delta)}}$ spans the univariate B-spline discretization space
\[
S[p,\Xi^{(k,\delta)}]:=\text{span}\{B^{(k,\delta)}[p,\Xi,1],\ldots, B^{(k,\delta)}[p,\Xi,n^{(k,\delta)}] \}.
\]
\begin{figure}[h]
	\begin{center}
		\includegraphics[trim=0 20px 0 0]{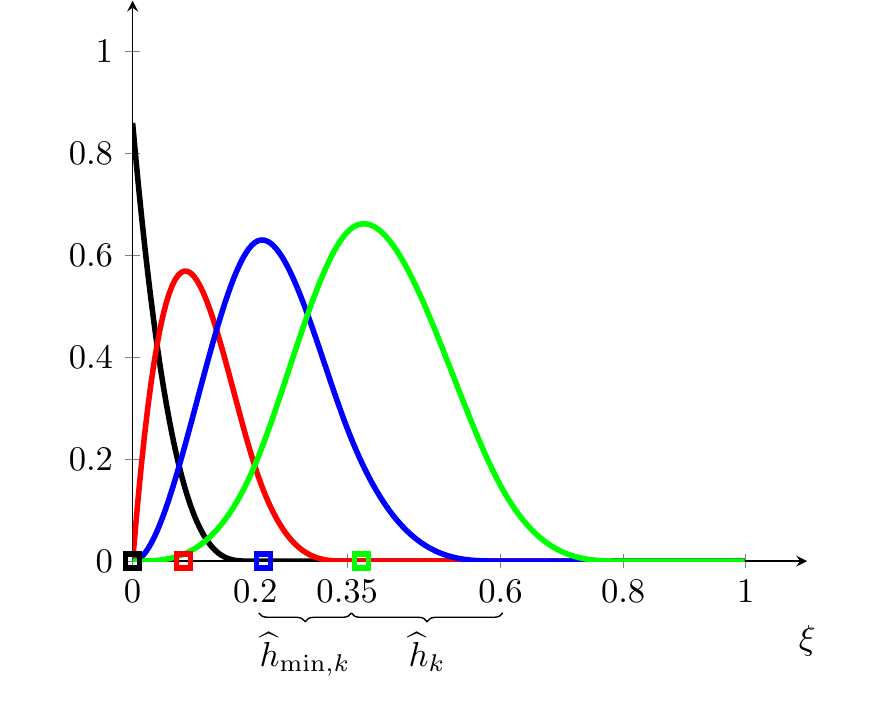}
	\end{center}
	\caption{Basis functions in univariate case\label{fig:basis1d}}
\end{figure}
We use the standard tensor-product B-spline space $\widehat{V}^{(k)}$ over $\widehat \Omega$, which is obtained from the tensor-product space of the two univariate spline spaces. The corresponding physical space  ${V}^{(k)}$ of $\widehat{V}^{(k)}$ is defined by the pull-back principle, i.e., 
\begin{equation}\label{eq:vkdef}
	\widehat{V}^{(k)}
	:= \{ v\in S[p,\Xi^{(k,1)}] \otimes S[p,\Xi^{(k,2)}]
	\;:\;
	v|_{\widehat \Gamma_D^{(k)}} = 0 \}	
	\quad\mbox{and}\quad
	V^{(k)} := \widehat{V}^{(k)} \circ G_k^{-1},
\end{equation}
where $v|_T$ denotes the restriction of $v$ to $T$ (trace operator).
The grid size $\widehat{h}_{k}$ on the parameter domain
and ${h}_{k}$ on the physical domain are defined by
\[
\widehat{h}_{k} :=
\max_{\delta=1,2}
\max
\{
\xi_{i+1}^{(k,\delta)} - \xi_i^{(k,\delta)}
\;:\;
i=1,\ldots,n^{(k,\delta)}+p
\}
\quad\mbox{and}\quad
h_{k}:=\widehat{h}_{k} \; \mbox{diam}(\Omega^{(k)}).
\]
Moreover, $\widehat{h}_{\mathrm{min},k}$ denotes the smallest knot
span, i.e., we define
\[
\widehat{h}_{\mathrm{min},k} :=
\min_{\delta=1,2}
\min
\{
\xi_{i+1}^{(k,\delta)} - \xi_i^{(k,\delta)}
\;:\;
i=1,\ldots,n^{(k,\delta)}+p
\;\mbox{where}\;
\xi_{i+1}^{(k,\delta)} \not= \xi_i^{(k,\delta)}
\}.
\]
Figure~\ref{fig:basis1d} shows the basis functions for the univariate
case. Here and in what follows, we identify each basis
function with its Greville point. We observe that there is only one
active basis function on each end point of the interval; its 
Greville point is located on that end point.

\begin{figure}
	\begin{center}
		\begin{tikzpicture}
			\def\shift{0.5}
			\fill[gray!20] (-0.2,0) -- (4.5+\shift,0) -- (4.5+\shift,1.7) -- (-0.2,1.7);
			\fill[gray!20] (-0.2,-1.5-\shift) -- (1.5,-1.5-\shift) -- (1.5,-3.2-\shift) -- (-0.2,-3.2-\shift);
			\fill[gray!20] (4.5+\shift,-1.5-\shift) -- (3.0+\shift,-1.5-\shift) -- (3.0+\shift,-3.2-\shift) -- (4.5+\shift,-3.2-\shift);
			\fill[gray!20] (5.7+\shift,1.7) -- (7.4+\shift,1.7) -- (7.4+\shift,0) -- (5.7+\shift,0);
			\fill[gray!20] (5.7+\shift,-3.2-\shift) -- (5.7+\shift,-1.5-\shift) -- (7.4+\shift,-1.5-\shift) -- (7.4+\shift,-3.2-\shift);
			
			\draw (-0.2,0) -- (4.5+\shift,0) -- (4.5+\shift,1.7) node at (2.5,2) {$\Omega^{(1)}$};
			\draw (-0.2,-1.5-\shift) -- (1.5, -1.5-\shift) -- (1.5,-3.2-\shift) node at (0.25,-3.5-\shift) {$\Omega^{(2)}$};
			\draw (3.0+\shift,-3.2-\shift) -- (3.0+\shift,-1.5-\shift) -- (4.5+\shift,-1.5-\shift) -- (4.5+\shift,-3.2-\shift) node at (4.75+\shift,-3.5-\shift) {};
			\node at (3.8,-3.5-\shift) {$\Omega^{(3)}$};
			\draw (5.7+\shift,1.7) -- (5.7+\shift,0) -- (7.4+\shift,0) node at (6.5+\shift,2) {$\Omega^{(4)}$};
			\draw (7.4+\shift,-1.5-\shift) -- (5.7+\shift,-1.5-\shift) -- (5.7+\shift,-3.2-\shift) node at (6.5+\shift,-3.5-\shift) {$\Omega^{(5)}$};
			
			\node at (-0.7,1.3) {$V^{(1)}$};		
			\node at (-0.7,-2.8-\shift) {$V^{(2)}$};
			\node at (4.8,-3.5-\shift) {$V^{(3)}$};
			\node at (8.4,1.3) {$V^{(4)}$};
			\node at (8.4,-2.8-\shift) {$V^{(5)}$};
			
			\draw (1.85,0) node[circle, fill, inner sep = 2.5pt] {};
			\draw (1.85,1.0) node[circle, fill, inner sep = 2.5pt] {};
			\draw (0.25,0.0) node[circle, fill, inner sep = 2.5pt] {};
			\draw (0.25,1.0) node[circle, fill, inner sep = 2.5pt] {};
			\draw (3.35,0.0) node[circle, fill, inner sep = 2.5pt] {};
			\draw (3.35,1.0) node[circle, fill, inner sep = 2.5pt] {};
			\draw (5.0,0.0) node[circle, fill, inner sep = 2.5pt] {};
			\draw (5.0,1.0) node[circle, fill, inner sep = 2.5pt] {};
			
			\draw (1.5,-1.5-\shift) node[diamond, fill, inner sep = 2pt] {};
			\draw (1.5,-2.2-\shift) node[diamond, fill, inner sep = 2pt] {};
			\draw (1.5,-3.0-\shift) node[diamond, fill, inner sep = 2pt] {};
			\draw (0.75,-1.5-\shift) node[diamond, fill, inner sep = 2pt] {};
			\draw (0.75,-2.2-\shift) node[diamond, fill, inner sep = 2pt] {};
			\draw (0.75,-3.0-\shift) node[diamond, fill, inner sep = 2pt] {};
			
			\draw (3.0+\shift,-1.5-\shift) node[star, fill, inner sep = 2pt] {};
			\draw (3.0+\shift,-2.0-\shift) node[star, fill, inner sep = 2pt] {};
			\draw (3.0+\shift,-2.75-\shift) node[star, fill, inner sep = 2pt] {};
			\draw (3.75+\shift,-1.5-\shift) node[star, fill, inner sep = 2pt] {};
			\draw (3.75+\shift,-2.0-\shift) node[star, fill, inner sep = 2pt] {};
			\draw (3.75+\shift,-2.75-\shift) node[star, fill, inner sep = 2pt] {};
			\draw (4.5+\shift,-1.5-\shift) node[star, fill, inner sep = 2pt] {};
			\draw (4.5+\shift,-2.0-\shift) node[star, fill, inner sep = 2pt] {};
			\draw (4.5+\shift,-2.75-\shift) node[star, fill, inner sep = 2pt] {};
			
			\draw (5.7+\shift,0) node[rectangle, fill, inner sep = 2pt] {};
			\draw (6.2+\shift,0) node[rectangle, fill, inner sep = 2pt] {};
			\draw (5.7+\shift,0.75) node[rectangle, fill, inner sep = 2pt] {};
			\draw (6.2+\shift,0.75) node[rectangle, fill, inner sep = 2pt] {};
			\draw (5.7+\shift,1.4) node[rectangle, fill, inner sep = 2pt] {};
			\draw (6.2+\shift,1.4) node[rectangle, fill, inner sep = 2pt] {};
			\draw (7.0+\shift,0) node[rectangle, fill, inner sep = 2pt] {};
			\draw (7.0+\shift,0.75) node[rectangle, fill, inner sep = 2pt] {};
			\draw (7.0+\shift,1.4) node[rectangle, fill, inner sep = 2pt] {};
			
			\draw (5.7+\shift,-1.5-\shift) node[regular polygon,regular polygon sides=3, fill, inner sep = 1.5pt] {};
			\draw (7.0+\shift,-1.5-\shift) node[regular polygon,regular polygon sides=3, fill, inner sep = 1.5pt] {};
			\draw (5.7+\shift,-3.-\shift) node[regular polygon,regular polygon sides=3, fill, inner sep = 1.5pt] {};
			\draw (7.0+\shift,-3.-\shift) node[regular polygon,regular polygon sides=3, fill, inner sep = 1.5pt] {};
			
		\end{tikzpicture}
		\captionof{figure}{Schematic representation of local spaces\label{fig:local}}
	\end{center}
\end{figure}
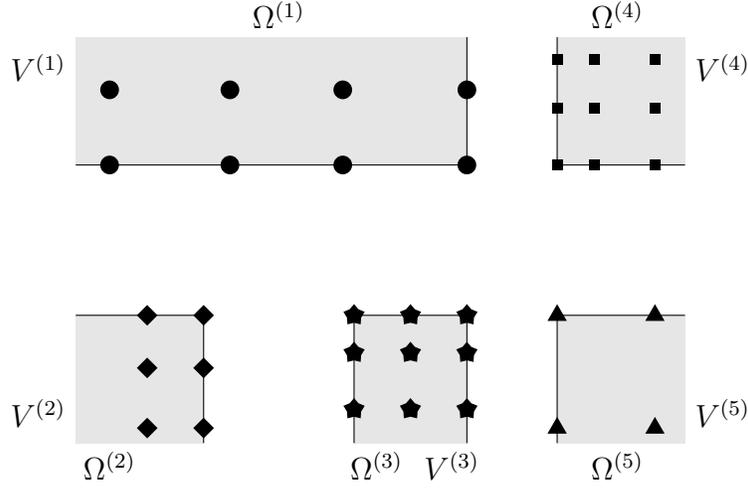
Consequently, the standard tensor-product basis is represented as
a grid of Greville points, cf. Figure~\ref{fig:local}, where an
example with five patches is depicted. Note that, in the physical domain, the patches adjoin directly. Since we employ a discontinuous Galerkin
discretization, the basis functions at the interfaces do not agree.
Therefore, we separate the patches visually. The patches $\Omega^{(1)}$,
$\Omega^{(3)}$, $\Omega^{(4)}$ and $\Omega^{(5)}$ meet in a regular
corner. Certainly, it is also possible that only three or more than
four patches meet in a regular corner.
Such junctions have been previously considered. Here, we additionally allow T-junctions, like the junction
between the patches $\Omega^{(1)}$, $\Omega^{(2)}$ and $\Omega^{(3)}$.
Certainly, it is possible that more than three patches meet in a T-junction. Note that in any case, the T-junction constitutes a corner
of every involved patch but one patch, like patch $\Omega^{(1)}$ in
the example.

Having all patchwise discretization spaces defined, we obtain the global approximation space by
\begin{equation}\label{eq:vdef}
	V := V^{(1)} \times \dots \times V^{(K)}.
\end{equation}
On this discontinuous discretization space, we introduce
a variational formulation of the model problem~\eqref{continousProb}
following the symmetric interior penalty discontinuous
Galerkin (SIPG) method. Find $u = (u^{(1)},\cdots, u^{(K)})
\in V$ such that
\begin{align}
	\label{discreteVarProb}
	a_h(u,v)=\langle f,v\rangle
	\quad \mbox{for all}\quad v \in V,
\end{align}
where 
\begin{align*}
	a_h(u,v) &:= \sum_{k=1}^{K} \left( a^{(k)}(u,v) + m^{(k)}(u,v) + r^{(k)}(u,v) \right), \\
	a^{(k)}(u,v) &:= \int_{\Omega^{(k)}} \alpha_k \nabla u^{(k)} \cdot \nabla v^{(k)} \; \textrm{d}x, \\
	m^{(k)}(u,v) &:= \sum_{\ell \in \mathcal{N}_\Gamma(k)} \int_{\Gamma^{(k,\ell)}} \frac{\alpha_k}{2} 
	\left( 
	\frac{\partial u^{(k)}}{\partial n_k}(v^{(\ell)} - v^{(k)}) + 
	\frac{\partial v^{(k)}}{\partial n_k}(u^{(\ell)} - u^{(k)})
	\right) \; \textrm{d}s, \\
	r^{(k)}(u,v) &:= \sum_{\ell \in \mathcal{N}_\Gamma(k)} \int_{\Gamma^{(k,\ell)}} \alpha_k \frac{\delta p^2}{\min\{h_k,h_\ell\}}   
	(u^{(\ell)} - u^{(k)})(v^{(\ell)} - v^{(k)}) \; \textrm{d}s,\\
	\langle {f}, v \rangle
	&:=  
	\sum_{k=1}^{K} \int_{\Omega^{(k)}} f v^{(k)} \; \textrm{d}x,
\end{align*}
and $\delta > 0$ is some suitably chosen penalty parameter and $n_k$ is the unit normal vector pointing outwards of the patch $\Omega^{(k)}$
and $\Gamma^{(k,\ell)}:= \partial \Omega^{(k)}\cap \partial \Omega^{(\ell)}$ is the interface between the two patches. $\mathcal{N}_\Gamma(k)$ contains the indices of the neighboring patches $\Omega^{(\ell)}$, sharing with $\Omega^{(k)}$ more than just a corner.

The penalty parameter $\delta$ ensures that the bilinear form $a_h(\cdot,\cdot)$ is bounded and coercive in the dG-norm
\[
\| v \|^2_{d} := d(v, v), \quad \text{ where } \quad
d(u,v) := \sum_{k=1}^{K} \left( a^{(k)}(u,v) + r^{(k)}(u,v) \right).
\]
A suitable $\delta$ can always be chosen, independently of the spline degree $p$ and grid sizes $h_k$, however it might depend on the
geometry functions and on the quasi-uniformity of the
grids, i.e., the ratios $\widehat h_k/\widehat h_{\mathrm{min},k}$,
see \cite[Theorem 8]{Takacs:2019b}. The Theorem of Lax-Milgram
guarantees the existence and the uniqueness of a solution to~\eqref{discreteVarProb}. If the solution $u$ of the continuous problem
is sufficiently smooth, the solution of \eqref{discreteVarProb} is
an approximation to the solution of the original problem \eqref{continousProb}, cf.~\cite[Theorems 12 and 13]{Takacs:2019b}.

\section{The dG IETI-DP solver}
\label{sec:4}
%
%
%

In this section, we introduce a IETI-DP solver for the discontinuous Galerkin discretization~\eqref{discreteVarProb}. The first step is the
introduction of the local subspaces required for the domain decomposition
method and the local assembling of the problem,
see Subsection~\ref{subs:4:1}. In Subsection~\ref{subs:4:2},
we discuss the introduction of the primal degrees of freedom. Then,
in Subsection~\ref{subs:4:3}, we discuss the coupling of the remaining
degrees of freedom using Lagrange multipliers. The setup
of the IETI system is discussed in Subsection~\ref{subs:4:4}, its solution
is discussed in Subsection~\ref{subs:4:5}, and the corresponding
convergence result is stated in Subsection~\ref{subs:4:6}.

\subsection{Local subspaces and local problem}\label{subs:4:1}

As it has been done in the seminal paper~\cite{KleissPechsteinJuttlerTomar:2012} and in follow-up publications
on IETI-DP, the local spaces are constructed on a per-patch basis. For
variational problems that are discretized using dG approaches, the setup
of local spaces is not obvious. We follow the approach that 
has been first introduced in~\cite{DryjaGalvis:2013} and then adapted to IgA in~\cite{HoferLanger:2016a,Hofer:2016a,SchneckenleitnerTakacs:2020,SchneckenleitnerTakacs:2021}:
The introduction of artificial interfaces.

The local function space $V^{(k)}_e$ for a patch $\Omega^{(k)}$ is
composed of the original function space $V^{(k)}$ of the patch and
of the traces of the function spaces $V^{(\ell)}$, when restricted
to the common interface $\Gamma^{(k,\ell)}$. Formally, we have
\begin{align*}
	V_{e}^{(k)}:= V^{(k)} \times \prod_{\ell \in \mathcal{N}_\Gamma(k)} V^{(k, \ell)},
	\quad \mbox{where}\quad
	V^{(k,\ell)} := \{ v^{(\ell)}|_{\Gamma^{(k,\ell)}} \;:\; v^{(\ell)}\in V^{(\ell)}\}.
\end{align*}
A local function $v_e^{(k)} \in V_e^{(k)}$ is represented as a tuple
\begin{align}
	\label{def:representation}
	v_e^{(k)} = 
	\left(
	v^{(k)}, (v^{(k,\ell)})_{\ell\in \mathcal{N}_\Gamma(k)}
	\right),
	\quad\mbox{where}\quad
	v^{(k)} \in V^{(k)}
	\mbox{ and }
	v^{(k,\ell)} \in V^{(k,\ell)}.
\end{align}

\begin{figure}
	\begin{center}
		\begin{tikzpicture}
			\def\shift{0.5}
			\fill[gray!20] (-0.2,0) -- (4.5+\shift,0) -- (4.5+\shift,1.7) -- (-0.2,1.7);
			\fill[gray!20] (-0.2,-1.5-\shift) -- (1.5,-1.5-\shift) -- (1.5,-3.2-\shift) -- (-0.2,-3.2-\shift);
			\fill[gray!20] (4.5+\shift,-1.5-\shift) -- (3.0+\shift,-1.5-\shift) -- (3.0+\shift,-3.2-\shift) -- (4.5+\shift,-3.2-\shift);
			\fill[gray!20] (6+\shift,1.7) -- (7.7+\shift,1.7) -- (7.7+\shift,0) -- (6+\shift,0);
			\fill[gray!20] (6+\shift,-3.2-\shift) -- (6+\shift,-1.5-\shift) -- (7.7+\shift,-1.5-\shift) -- (7.7+\shift,-3.2-\shift);
			
			\draw (-0.2,0) -- (4.5+\shift,0) -- (4.5+\shift,1.7) node at (2.5,2) {$\Omega^{(1)}$};
			\draw (-0.2,-1.5-\shift) -- (1.5, -1.5-\shift) -- (1.5,-3.2-\shift) node at (0.25,-3.5-\shift) {$\Omega^{(2)}$};
			\draw (3.0+\shift,-3.2-\shift) -- (3.0+\shift,-1.5-\shift) -- (4.5+\shift,-1.5-\shift) -- (4.5+\shift,-3.2-\shift) node at (4.+\shift,-3.5-\shift) {$\Omega^{(3)}$};
			\draw (6+\shift,1.7) -- (6+\shift,0) -- (7.7+\shift,0) node at (7.0+\shift,2) {$\Omega^{(4)}$};
			\draw (7.7+\shift,-1.5-\shift) -- (6+\shift,-1.5-\shift) -- (6+\shift,-3.2-\shift) node at (7+\shift,-3.5-\shift) {$\Omega^{(5)}$};
			
			\node at (-0.7,1.3) {$V^{(1)}$};
			\node at (-0.7, -0.4) {$V^{(1,2)}$};
			\node at (4+\shift,-0.75) {$V^{(1,3)}$};
			\node at (4.5+\shift+0.4,2) {$V^{(1,4)}$};
			
			\node at (-0.7,-2.8-\shift) {$V^{(2)}$};
			\node at (-0.7,-1.1-\shift) {$V^{(2,1)}$};
			\node at (2.0,-3.5-\shift) {$V^{(2,3)}$};
			
			\node at (2.6+\shift,-3.5-\shift) {$V^{(3,2)}$};	
			\node at (4.+\shift,-0.75-\shift) {$V^{(3,1)}$};
			\node at (4.5+\shift+0.4,-3.5-\shift) {$V^{(3,5)}$};
			
			\node at (8.7,1.3) {$V^{(4)}$};
			\node at (8.7,-0.4) {$V^{(4,5)}$};
			\node at (6+\shift,2) {$V^{(4,1)}$};
			
			\node at (8.7,-2.8-\shift) {$V^{(5)}$};
			\node at (8.7,-1.1-\shift) {$V^{(5,4)}$};
			\node at (6+\shift,-3.5-\shift) {$V^{(5,3)}$};
			
			\draw (-0.2,-0.4) -- (1.5,-0.4) {};
			\draw (3.0+\shift,-0.4) -- (4.5+\shift,-0.4) {};
			\draw (4.5+\shift+0.4,0) -- (4.5+\shift+0.4,1.7) {};
			
			\draw (-0.2,-1.1-\shift) -- (1.85,-1.1-\shift) {};
			\draw (1.9,-1.5-\shift) -- (1.9,-3.2-\shift) {};
			
			\draw (3.35,-1.1-\shift) -- (4.5+\shift,-1.1-\shift) {};
			\draw (2.6+\shift,-1.5-\shift) -- (2.6+\shift,-3.2-\shift) {};	
			\draw (4.5+\shift+0.4,-1.5-\shift) -- (4.5+\shift+0.4,-3.2-\shift) {};	
			
			\draw (6+\shift,-1.1-\shift) -- (7.7+\shift,-1.1-\shift) {};
			\draw (6+\shift-0.4,-1.5-\shift) -- (6+\shift-0.4,-3.2-\shift) {};	
			
			\draw (6+\shift,-0.4) -- (7.7+\shift,-0.4) {};
			\draw (6+\shift-0.4,0) -- (6+\shift-0.4,1.7) {};		
			
			\draw (1.85,0) node[circle, fill, inner sep = 2.5pt] {};
			\draw (1.85,1.0) node[circle, fill, inner sep = 2.5pt] {};
			\draw (0.25,0.0) node[circle, fill, inner sep = 2.5pt] {};
			\draw (0.25,1.0) node[circle, fill, inner sep = 2.5pt] {};
			\draw (3.35,0.0) node[circle, fill, inner sep = 2.5pt] {};
			\draw (3.35,1.0) node[circle, fill, inner sep = 2.5pt] {};
			\draw (5.0,0.0) node[circle, fill, inner sep = 2.5pt] {};
			\draw (5.0,1.0) node[circle, fill, inner sep = 2.5pt] {};
			\draw (3.0+\shift,-0.4) node[star, fill, inner sep = 2pt] {};
			\draw (3.75+\shift,-0.4) node[star, fill, inner sep = 2pt] {};
			\draw (4.5+\shift,-0.4) node[star, fill, inner sep = 2pt] {};
			\draw (0.75,-0.4) node[diamond, fill, inner sep = 2pt] {};
			\draw (1.5,-0.4) node[diamond, fill, inner sep = 2pt] {};
			\draw (4.5+\shift+0.4,0) node[rectangle, fill, inner sep = 2pt] {};
			\draw (4.5+\shift+0.4,0.75) node[rectangle, fill, inner sep = 2pt] {};
			\draw (4.5+\shift+0.4,1.4) node[rectangle, fill, inner sep = 2pt] {};
			
			\draw (1.5,-1.5-\shift) node[diamond, fill, inner sep = 2pt] {};
			\draw (1.5,-2.2-\shift) node[diamond, fill, inner sep = 2pt] {};
			\draw (1.5,-3.0-\shift) node[diamond, fill, inner sep = 2pt] {};
			\draw (0.75,-1.5-\shift) node[diamond, fill, inner sep = 2pt] {};
			\draw (0.75,-2.2-\shift) node[diamond, fill, inner sep = 2pt] {};
			\draw (0.75,-3.0-\shift) node[diamond, fill, inner sep = 2pt] {};
			
			\draw (1.9,-1.5-\shift) node[star, fill, inner sep = 2pt] {};
			\draw (1.9,-2.0-\shift) node[star, fill, inner sep = 2pt] {};
			\draw (1.9,-2.75-\shift) node[star, fill, inner sep = 2pt] {};
			\draw (0.25,-1.1-\shift) node[circle, fill, inner sep = 2.5pt] {};
			\draw (1.85,-1.1-\shift) node[circle, fill, inner sep = 2.5pt] {};
			\draw (1.85,-0.5-\shift) node[circle, fill, inner sep = 2.5pt] {};
			
			\draw (3.0+\shift,-1.5-\shift) node[star, fill, inner sep = 2pt] {};
			\draw (3.0+\shift,-2.0-\shift) node[star, fill, inner sep = 2pt] {};
			\draw (3.0+\shift,-2.75-\shift) node[star, fill, inner sep = 2pt] {};
			\draw (3.75+\shift,-1.5-\shift) node[star, fill, inner sep = 2pt] {};
			\draw (3.75+\shift,-2.0-\shift) node[star, fill, inner sep = 2pt] {};
			\draw (3.75+\shift,-2.75-\shift) node[star, fill, inner sep = 2pt] {};
			\draw (4.5+\shift,-1.5-\shift) node[star, fill, inner sep = 2pt] {};
			\draw (4.5+\shift,-2.0-\shift) node[star, fill, inner sep = 2pt] {};
			\draw (4.5+\shift,-2.75-\shift) node[star, fill, inner sep = 2pt] {};
			
			\draw (2.6+\shift,-1.5-\shift) node[diamond, fill, inner sep = 2pt] {};
			\draw (2.6+\shift,-2.2-\shift) node[diamond, fill, inner sep = 2pt] {};
			\draw (2.6+\shift,-3.0-\shift) node[diamond, fill, inner sep = 2pt] {};
			\draw (3.35,-0.5-\shift) node[circle, fill, inner sep = 2.5pt] {};
			\draw (3.35,-1.1-\shift) node[circle, fill, inner sep = 2.5pt] {};
			\draw (4.5+\shift,-1.1-\shift) node[circle, fill, inner sep = 2.5pt] {};
			\draw (4.5+\shift+0.4,-1.5-\shift) node[regular polygon,regular polygon sides=3, fill, inner sep = 1.5pt] {};
			\draw (4.5+\shift+0.4,-3-\shift) node[regular polygon,regular polygon sides=3, fill, inner sep = 1.5pt] {};
			
			\draw (6+\shift,0) node[rectangle, fill, inner sep = 2pt] {};
			\draw (6.5+\shift,0) node[rectangle, fill, inner sep = 2pt] {};
			\draw (6+\shift,0.75) node[rectangle, fill, inner sep = 2pt] {};
			\draw (6.5+\shift,0.75) node[rectangle, fill, inner sep = 2pt] {};
			\draw (6+\shift,1.4) node[rectangle, fill, inner sep = 2pt] {};
			\draw (6.5+\shift,1.4) node[rectangle, fill, inner sep = 2pt] {};
			\draw (7.3+\shift,0) node[rectangle, fill, inner sep = 2pt] {};
			\draw (7.3+\shift,0.75) node[rectangle, fill, inner sep = 2pt] {};
			\draw (7.3+\shift,1.4) node[rectangle, fill, inner sep = 2pt] {};
			\draw (6+\shift-0.4,0) node[circle, fill, inner sep = 2.5pt] {};
			\draw (6+\shift-0.4,1.) node[circle, fill, inner sep = 2.5pt] {};
			\draw (6+\shift,-0.4) node[regular polygon,regular polygon sides=3, fill, inner sep = 1.5pt] {};
			\draw (7.3+\shift,-0.4) node[regular polygon,regular polygon sides=3, fill, inner sep = 1.5pt] {};
			
			\draw (6+\shift,-1.5-\shift) node[regular polygon,regular polygon sides=3, fill, inner sep = 1.5pt] {};
			\draw (7.3+\shift,-1.5-\shift) node[regular polygon,regular polygon sides=3, fill, inner sep = 1.5pt] {};
			\draw (6+\shift,-3.-\shift) node[regular polygon,regular polygon sides=3, fill, inner sep = 1.5pt] {};
			\draw (7.3+\shift,-3.-\shift) node[regular polygon,regular polygon sides=3, fill, inner sep = 1.5pt] {};
			\draw (6+\shift-0.4,-1.5-\shift) node[star, fill, inner sep = 2pt] {};
			\draw (6+\shift-0.4,-2.-\shift) node[star, fill, inner sep = 2pt] {};
			\draw (6+\shift-0.4,-2.75-\shift) node[star, fill, inner sep = 2pt] {};
			\draw (6+\shift,-1.1-\shift) node[rectangle, fill, inner sep = 2pt] {};
			\draw (6.5+\shift,-1.1-\shift) node[rectangle, fill, inner sep = 2pt] {};
			\draw (7.3+\shift,-1.1-\shift) node[rectangle, fill, inner sep = 2pt] {};
			
			\draw (1.85,-1.1-\shift) -- (2.30,-1.1-\shift) {};
			\draw (2.30,-1.1-\shift) -- (2.90,-0.5-\shift) {};
			\draw (2.90,-0.5-\shift) -- (3.35,-0.5-\shift) {};
			
			\draw (1.85,-0.5-\shift) -- (2.30,-0.5-\shift) {};
			\draw (2.30,-0.5-\shift) -- (2.54,-0.74-\shift) {};
			\draw (2.66,-0.86-\shift) -- (2.90,-1.1-\shift) {};
			\draw (2.90,-1.1-\shift) -- (3.35,-1.1-\shift) {};

		\end{tikzpicture}
		\captionof{figure}{Schematic representation of local spaces with artificial interfaces \label{fig:ai}}
	\end{center}
\end{figure}
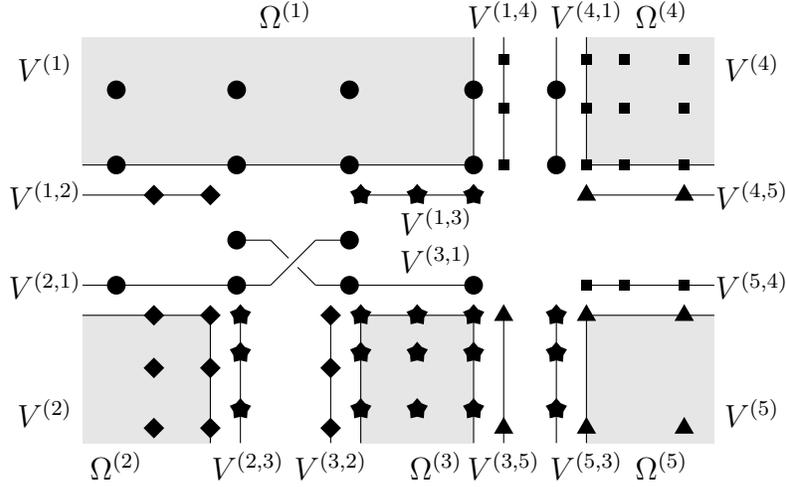

The discretization space is visualized in Figure~\ref{fig:ai}. Again,
we depict the the interfaces and artificial interfaces separately
since there live different function spaces.
We again represent every basis function with its Greville point.
Basis functions that origin from the same basis are denoted by the same symbol. 
\begin{figure}[h]
	\begin{center}
		\includegraphics[trim=0 20px 0 0]{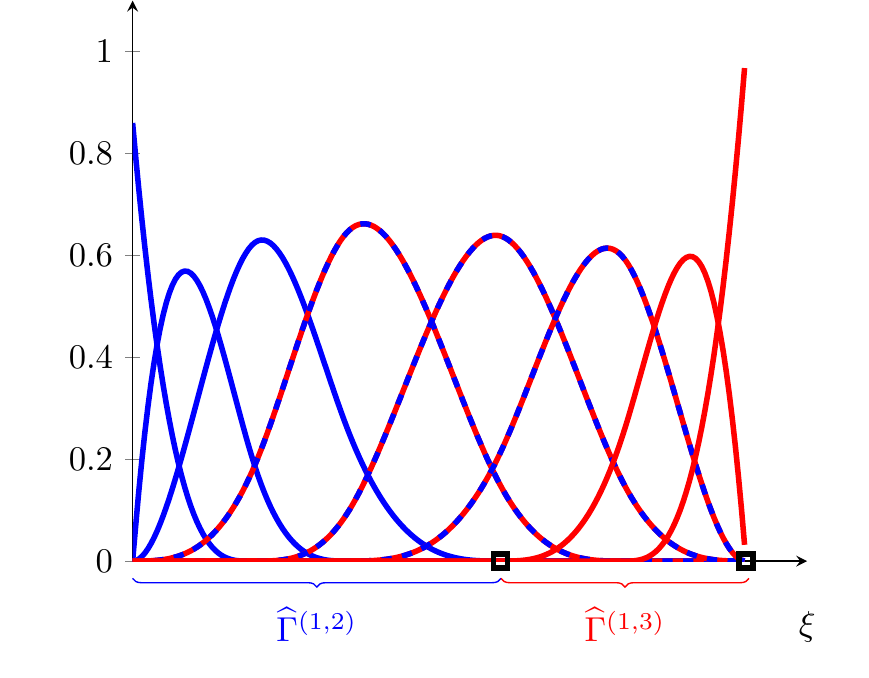}
	\end{center}
	\caption{Basis functions selected for artificial interfaces; the dashed basis functions are selected for both bases\label{fig:arti_basis}}
\end{figure}
The function spaces for the
artificial interfaces, like $V^{(2,1)}$, are the traces of the
corresponding spaces, here $V^{(1)}$. Their basis just consists of
the traces of those basis functions of the basis of $V^{(1)}$ that are
active on the interface $\Gamma^{(1,2)}$. While this is rather obvious
for the case of interfaces that span over a whole edge, it needs some
more elaboration in the context of T-junctions. As one can see in
Figure~\ref{fig:arti_basis}, basis functions on the bottom side of the
patch $\Omega^{(1)}$ are chosen to be part of the artificial interface
if they do not vanish on $\Gamma^{(1,2)}$. This includes basis functions
whose Greville point is not located on $\Gamma^{(1,2)}$. The corresponding basis functions form a part of both, the basis for
$V^{(2,1)}$ and the basis for $V^{(3,1)}$. In Figure~\ref{fig:ai}, we depict these degrees of
freedom on extensions of the artificial interfaces.

For the IETI formulation, we collect the spaces that share the first
index together, so for example $V_e^{(1)}$ is composed of
$V^{(1)}$, $V^{(1,2)}$, $V^{(1,3)}$, and $V^{(1,4)}$.

We introduce local bilinear forms $a_e^{(k)}(\cdot,\cdot)$ and
$d_e^{(k)}(\cdot,\cdot)$ and the local linear functional $\langle f_e^{(k)},\cdot\rangle$ that live on the spaces $V_e^{(k)}$.
They can be seen as the local counterparts to
$a_h(\cdot,\cdot)$, $d(\cdot,\cdot)$, and $\langle f,\cdot\rangle$ and we define them by 
\[
\begin{aligned}
	a_e^{(k)}(u_e^{(k)},v_e^{(k)}) &:= a^{(k)}(u_e^{(k)},v_e^{(k)}) + m^{(k)}(u_e^{(k)},v_e^{(k)}) + r^{(k)}(u_e^{(k)},v_e^{(k)}), \\
	d_e^{(k)}(u_e^{(k)},v_e^{(k)}) &:= a^{(k)}(u_e^{(k)},v_e^{(k)}) + r^{(k)}(u_e^{(k)},v_e^{(k)}), \\
	\langle f_e^{(k)},v_e^{(k)}\rangle & := \int_{\Omega^{(k)}} f v^{(k)} \mathrm dx,
\end{aligned}
\]
where we write with a slight abuse of notation
\[
\begin{aligned}
	a^{(k)}(u_e^{(k)},v_e^{(k)}) &:= \int_{\Omega^{(k)}} \alpha_k \nabla u^{(k)} \cdot \nabla v^{(k)} \; \textrm{d}x, \\
	m^{(k)}(u_e^{(k)},v_e^{(k)}) &:= \sum_{\ell \in \mathcal{N}_\Gamma(k)} \int_{\Gamma^{(k,\ell)}} \frac{\alpha_k}{2} 
	\left( 
	\frac{\partial u^{(k)}}{\partial n_k}(v^{(k,\ell)} - v^{(k)}) + 
	\frac{\partial v^{(k)}}{\partial n_k}(u^{(k,\ell)} - u^{(k)})
	\right) \; \textrm{d}s, \\
	r^{(k)}(u_e^{(k)},v_e^{(k)}) &:= \sum_{\ell \in \mathcal{N}_\Gamma(k)} \int_{\Gamma^{(k,\ell)}} \alpha_k \frac{\delta p^2}{\min\{h_{k},h_{\ell}\}}   
	(u^{(k,\ell)} - u^{(k)})(v^{(k,\ell)} - v^{(k)}) \; \textrm{d}s.
\end{aligned}
\]

The discretization of $a_e^{(k)}(\cdot,\cdot)$ and
$\langle f_e^{(k)},\cdot \rangle$
with respect to the chosen basis for $V_e^{(k)}$ gives the local
linear system
\begin{equation}\label{linsys:local}
	A^{(k)}\, \underline{u}_e^{(k)} = \underline{f}_e^{(k)}.
\end{equation}

\subsection{Primal degrees of freedom}\label{subs:4:2}

\begin{figure}[h]
	\begin{center}
		\begin{tikzpicture}
			\def\shift{0.5}
			\fill[gray!20] (-0.2,0) -- (4.5+\shift,0) -- (4.5+\shift,1.7) -- (-0.2,1.7);
			\fill[gray!20] (-0.2,-1.5-\shift) -- (1.5,-1.5-\shift) -- (1.5,-3.2-\shift) -- (-0.2,-3.2-\shift);
			\fill[gray!20] (4.5+\shift,-1.5-\shift) -- (3.0+\shift,-1.5-\shift) -- (3.0+\shift,-3.2-\shift) -- (4.5+\shift,-3.2-\shift);
			\fill[gray!20] (6+\shift,1.7) -- (7.7+\shift,1.7) -- (7.7+\shift,0) -- (6+\shift,0);
			\fill[gray!20] (6+\shift,-3.2-\shift) -- (6+\shift,-1.5-\shift) -- (7.7+\shift,-1.5-\shift) -- (7.7+\shift,-3.2-\shift);
			
			\draw (-0.2,0) -- (4.5+\shift,0) -- (4.5+\shift,1.7) node at (2.5,2) {$\Omega^{(1)}$};
			\draw (-0.2,-1.5-\shift) -- (1.5, -1.5-\shift) -- (1.5,-3.2-\shift) node at (0.25,-3.5-\shift) {$\Omega^{(2)}$};
			\draw (3.0+\shift,-3.2-\shift) -- (3.0+\shift,-1.5-\shift) -- (4.5+\shift,-1.5-\shift) -- (4.5+\shift,-3.2-\shift) node at (4.+\shift,-3.5-\shift) {$\Omega^{(3)}$};
			\draw (6+\shift,1.7) -- (6+\shift,0) -- (7.7+\shift,0) node at (7.0+\shift,2) {$\Omega^{(4)}$};
			\draw (7.7+\shift,-1.5-\shift) -- (6+\shift,-1.5-\shift) -- (6+\shift,-3.2-\shift) node at (7+\shift,-3.5-\shift) {$\Omega^{(5)}$};
		
			\draw (-0.2,-0.4) -- (1.5,-0.4) {};
			\draw (3.0+\shift,-0.4) -- (4.5+\shift,-0.4) {};
			\draw (4.5+\shift+0.4,0) -- (4.5+\shift+0.4,1.7) {};
			
			\draw (-0.2,-1.1-\shift) -- (1.85,-1.1-\shift) {};
			\draw (1.9,-1.5-\shift) -- (1.9,-3.2-\shift) {};
			
			\draw (3.35,-1.1-\shift) -- (4.5+\shift,-1.1-\shift) {};
			\draw (2.6+\shift,-1.5-\shift) -- (2.6+\shift,-3.2-\shift) {};	
			\draw (4.5+\shift+0.4,-1.5-\shift) -- (4.5+\shift+0.4,-3.2-\shift) {};	
			
			\draw (6+\shift,-1.1-\shift) -- (7.7+\shift,-1.1-\shift) {};
			\draw (6+\shift-0.4,-1.5-\shift) -- (6+\shift-0.4,-3.2-\shift) {};	
			
			\draw (6+\shift,-0.4) -- (7.7+\shift,-0.4) {};
			\draw (6+\shift-0.4,0) -- (6+\shift-0.4,1.7) {};		
			
			\draw (1.85,0) node[circle, fill, inner sep = 2.5pt] (A1) {};
			\draw (1.85,1.0) node[circle, fill, inner sep = 2.5pt] {};
			\draw (0.25,0.0) node[circle, fill, inner sep = 2.5pt] {};
			\draw (0.25,1.0) node[circle, fill, inner sep = 2.5pt] {};
			\draw (3.35,0.0) node[circle, fill, inner sep = 2.5pt] (A2) {};
			\draw (3.35,1.0) node[circle, fill, inner sep = 2.5pt] {};
			\draw (5.0,0.0) node[circle, fill, inner sep = 2.5pt] (A3) {};
			\draw (5.0,1.0) node[circle, fill, inner sep = 2.5pt] {};
			\draw (3.0+\shift,-0.4) node[star, fill, inner sep = 2pt] (A4) {};
			\draw (3.75+\shift,-0.4) node[star, fill, inner sep = 2pt] {};
			\draw (4.5+\shift,-0.4) node[star, fill, inner sep = 2pt] (A5) {};
			\draw (0.75,-0.4) node[diamond, fill, inner sep = 2pt] {};
			\draw (1.5,-0.4) node[diamond, fill, inner sep = 2pt] (A6) {};
			\draw (4.5+\shift+0.4,0) node[rectangle, fill, inner sep = 2pt] (A7) {};
			\draw (4.5+\shift+0.4,0.75) node[rectangle, fill, inner sep = 2pt] {};
			\draw (4.5+\shift+0.4,1.4) node[rectangle, fill, inner sep = 2pt] {};
			
			\draw (1.5,-1.5-\shift) node[diamond, fill, inner sep = 2pt] (B1) {};
			\draw (1.5,-2.2-\shift) node[diamond, fill, inner sep = 2pt] {};
			\draw (1.5,-3.0-\shift) node[diamond, fill, inner sep = 2pt] {};
			\draw (0.75,-1.5-\shift) node[diamond, fill, inner sep = 2pt] {};
			\draw (0.75,-2.2-\shift) node[diamond, fill, inner sep = 2pt] {};
			\draw (0.75,-3.0-\shift) node[diamond, fill, inner sep = 2pt] {};
			\draw (1.9,-1.5-\shift) node[star, fill, inner sep = 2pt] (B2) {};
			\draw (1.9,-2.0-\shift) node[star, fill, inner sep = 2pt] {};
			\draw (1.9,-2.75-\shift) node[star, fill, inner sep = 2pt] {};
			\draw (0.25,-1.1-\shift) node[circle, fill, inner sep = 2.5pt] {};
			\draw (1.85,-1.1-\shift) node[circle, fill, inner sep = 2.5pt] (B3) {};
			\draw (1.85,-0.5-\shift) node[circle, fill, inner sep = 2.5pt] (C7) {};
			
			\draw (3.0+\shift,-1.5-\shift) node[star, fill, inner sep = 2pt] (C1) {};
			\draw (3.0+\shift,-2.0-\shift) node[star, fill, inner sep = 2pt] {};
			\draw (3.0+\shift,-2.75-\shift) node[star, fill, inner sep = 2pt] {};
			\draw (3.75+\shift,-1.5-\shift) node[star, fill, inner sep = 2pt] {};
			\draw (3.75+\shift,-2.0-\shift) node[star, fill, inner sep = 2pt] {};
			\draw (3.75+\shift,-2.75-\shift) node[star, fill, inner sep = 2pt] {};
			\draw (4.5+\shift,-1.5-\shift) node[star, fill, inner sep = 2pt] (C2) {};
			\draw (4.5+\shift,-2.0-\shift) node[star, fill, inner sep = 2pt] {};
			\draw (4.5+\shift,-2.75-\shift) node[star, fill, inner sep = 2pt] {};
			\draw (2.6+\shift,-1.5-\shift) node[diamond, fill, inner sep = 2pt] (C3) {};
			\draw (2.6+\shift,-2.2-\shift) node[diamond, fill, inner sep = 2pt] {};
			\draw (2.6+\shift,-3.0-\shift) node[diamond, fill, inner sep = 2pt] {};
			\draw (3.35,-0.5-\shift) node[circle, fill, inner sep = 2.5pt] (B4) {};
			\draw (3.35,-1.1-\shift) node[circle, fill, inner sep = 2.5pt] (C4) {};
			\draw (4.5+\shift,-1.1-\shift) node[circle, fill, inner sep = 2.5pt] (C5) {};
			\draw (4.5+\shift+0.4,-1.5-\shift) node[regular polygon,regular polygon sides=3, fill, inner sep = 1.5pt] (C6) {};
			\draw (4.5+\shift+0.4,-3-\shift) node[regular polygon,regular polygon sides=3, fill, inner sep = 1.5pt] {};
			
			\draw (6+\shift,0) node[rectangle, fill, inner sep = 2pt] (D1) {};
			\draw (6.5+\shift,0) node[rectangle, fill, inner sep = 2pt] {};
			\draw (6+\shift,0.75) node[rectangle, fill, inner sep = 2pt] {};
			\draw (6.5+\shift,0.75) node[rectangle, fill, inner sep = 2pt] {};
			\draw (6+\shift,1.4) node[rectangle, fill, inner sep = 2pt] {};
			\draw (6.5+\shift,1.4) node[rectangle, fill, inner sep = 2pt] {};
			\draw (7.3+\shift,0) node[rectangle, fill, inner sep = 2pt] {};
			\draw (7.3+\shift,0.75) node[rectangle, fill, inner sep = 2pt] {};
			\draw (7.3+\shift,1.4) node[rectangle, fill, inner sep = 2pt] {};
			\draw (6+\shift-0.4,0) node[circle, fill, inner sep = 2.5pt] (D2) {};
			\draw (6+\shift-0.4,1.) node[circle, fill, inner sep = 2.5pt] {};
			\draw (6+\shift,-0.4) node[regular polygon,regular polygon sides=3, fill, inner sep = 1.5pt] (D3) {};
			\draw (7.3+\shift,-0.4) node[regular polygon,regular polygon sides=3, fill, inner sep = 1.5pt] {};
			
			\draw (6+\shift,-1.5-\shift) node[regular polygon,regular polygon sides=3, fill, inner sep = 1.5pt] (E1) {};
			\draw (7.3+\shift,-1.5-\shift) node[regular polygon,regular polygon sides=3, fill, inner sep = 1.5pt] {};
			\draw (6+\shift,-3.-\shift) node[regular polygon,regular polygon sides=3, fill, inner sep = 1.5pt] {};
			\draw (7.3+\shift,-3.-\shift) node[regular polygon,regular polygon sides=3, fill, inner sep = 1.5pt] {};
			\draw (6+\shift-0.4,-1.5-\shift) node[star, fill, inner sep = 2pt] (E2) {};
			\draw (6+\shift-0.4,-2.-\shift) node[star, fill, inner sep = 2pt] {};
			\draw (6+\shift-0.4,-2.75-\shift) node[star, fill, inner sep = 2pt] {};
			\draw (6+\shift,-1.1-\shift) node[rectangle, fill, inner sep = 2pt] (E3) {};
			\draw (6.5+\shift,-1.1-\shift) node[rectangle, fill, inner sep = 2pt] {};
			\draw (7.3+\shift,-1.1-\shift) node[rectangle, fill, inner sep = 2pt] {};
			
			\draw (1.85,-1.1-\shift) -- (2.30,-1.1-\shift) {};
			\draw (2.30,-1.1-\shift) -- (2.90,-0.5-\shift) {};
			\draw (2.90,-0.5-\shift) -- (3.35,-0.5-\shift) {};
			
			\draw (1.85,-0.5-\shift) -- (2.30,-0.5-\shift) {};
			\draw (2.30,-0.5-\shift) -- (2.54,-0.74-\shift) {};
			\draw (2.66,-0.86-\shift) -- (2.90,-1.1-\shift) {};
			\draw (2.90,-1.1-\shift) -- (3.35,-1.1-\shift) {};

			\draw[<->, line width = 1pt, latex-latex, bend right]
			(C7) edge (A1) (A2) edge (B4)
			(C7) edge (B3) (C4) edge (B4)
			(A6) edge (B1) (C2) edge (A5)
			(C2) edge (E2) (E3) edge (D1);
			
			\draw[<->, line width = 1pt, latex-latex, bend left]
			(A4) edge (C1) (B1) edge (C3) 
			(C1) edge (B2) (D1) edge (A7) 
			(A3) edge (D2) (C6) edge (E1)
			(E1) edge (D3) (C5) edge (A3);
		\end{tikzpicture}
		\captionof{figure}{Primal constraints \label{fig:primal}}
	\end{center}
\end{figure}
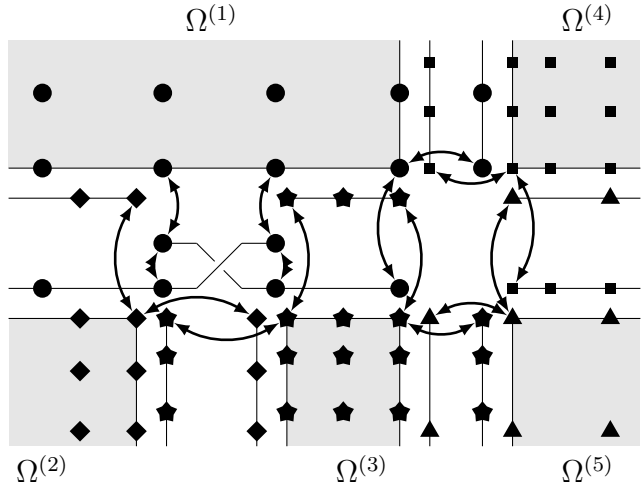

Concerning the choice of the primal degrees of freedom, we follow
the idea of vertex values. If a vertex $\textbf x$ happens to be located
on a corner of a patch $\Omega^{(k)}$, on each of the corresponding patches,
there is only one basis
function active. In this case, the primal constraint enforces
\[
u^{(k)} (\textbf x) = u^{(\ell,k)}(\textbf x)
\]
for all neighbors $\Omega^{(\ell)}$ that share the the vertex $\textbf x$. The primal
constraint is enforced by requiring that the coefficients for the
corresponding basis function agree.

If a vertex happens to be a T-junction $\textbf x$, we select all
basis functions which do not vanish on the T-junction. All corresponding
degrees of freedom are then treated as primal degrees of freedom. The
primal constraint is again enforced by requiring that the coefficients
for the corresponding basis functions agree.

The setup of the primal degrees of freedom is visualized in
Figure~\ref{fig:primal}.

\subsection{Jump matrices and Lagrange multipliers}\label{subs:4:3}

\begin{figure}[h]
	\begin{center}
		\begin{tikzpicture}
			\def\shift{0.5}
			\fill[gray!20] (-0.2,0) -- (4.5+\shift,0) -- (4.5+\shift,1.7) -- (-0.2,1.7);
			\fill[gray!20] (-0.2,-1.5-\shift) -- (1.5,-1.5-\shift) -- (1.5,-3.2-\shift) -- (-0.2,-3.2-\shift);
			\fill[gray!20] (4.5+\shift,-1.5-\shift) -- (3.0+\shift,-1.5-\shift) -- (3.0+\shift,-3.2-\shift) -- (4.5+\shift,-3.2-\shift);
			\fill[gray!20] (6+\shift,1.7) -- (7.7+\shift,1.7) -- (7.7+\shift,0) -- (6+\shift,0);
			\fill[gray!20] (6+\shift,-3.2-\shift) -- (6+\shift,-1.5-\shift) -- (7.7+\shift,-1.5-\shift) -- (7.7+\shift,-3.2-\shift);
			
			\draw (-0.2,0) -- (4.5+\shift,0) -- (4.5+\shift,1.7) node at (2.5,2) {$\Omega^{(1)}$};
			\draw (-0.2,-1.5-\shift) -- (1.5, -1.5-\shift) -- (1.5,-3.2-\shift) node at (0.25,-3.5-\shift) {$\Omega^{(2)}$};
			\draw (3.0+\shift,-3.2-\shift) -- (3.0+\shift,-1.5-\shift) -- (4.5+\shift,-1.5-\shift) -- (4.5+\shift,-3.2-\shift) node at (4.+\shift,-3.5-\shift) {$\Omega^{(3)}$};
			\draw (6+\shift,1.7) -- (6+\shift,0) -- (7.7+\shift,0) node at (7.0+\shift,2) {$\Omega^{(4)}$};
			\draw (7.7+\shift,-1.5-\shift) -- (6+\shift,-1.5-\shift) -- (6+\shift,-3.2-\shift) node at (7+\shift,-3.5-\shift) {$\Omega^{(5)}$};
			
			\draw (-0.2,-0.4) -- (1.5,-0.4) {};
			\draw (3.0+\shift,-0.4) -- (4.5+\shift,-0.4) {};
			\draw (4.5+\shift+0.4,0) -- (4.5+\shift+0.4,1.7) {};
			
			\draw (-0.2,-1.1-\shift) -- (1.85,-1.1-\shift) {};
			\draw (1.9,-1.5-\shift) -- (1.9,-3.2-\shift) {};
			
			\draw (3.35,-1.1-\shift) -- (4.5+\shift,-1.1-\shift) {};
			\draw (2.6+\shift,-1.5-\shift) -- (2.6+\shift,-3.2-\shift) {};	
			\draw (4.5+\shift+0.4,-1.5-\shift) -- (4.5+\shift+0.4,-3.2-\shift) {};	
			
			\draw (6+\shift,-1.1-\shift) -- (7.7+\shift,-1.1-\shift) {};
			\draw (6+\shift-0.4,-1.5-\shift) -- (6+\shift-0.4,-3.2-\shift) {};	
			
			\draw (6+\shift,-0.4) -- (7.7+\shift,-0.4) {};
			\draw (6+\shift-0.4,0) -- (6+\shift-0.4,1.7) {};		
			
			\draw (1.85,0) node[circle, fill, inner sep = 2.5pt]  {};
			\draw (1.85,1.0) node[circle, fill, inner sep = 2.5pt] {};
			\draw (0.25,0.0) node[circle, fill, inner sep = 2.5pt] (A1) {};
			\draw (0.25,1.0) node[circle, fill, inner sep = 2.5pt] {};
			\draw (3.35,0.0) node[circle, fill, inner sep = 2.5pt] {};
			\draw (3.35,1.0) node[circle, fill, inner sep = 2.5pt] {};
			\draw (5.0,0.0) node[circle, fill, inner sep = 2.5pt] {};
			\draw (5.0,1.0) node[circle, fill, inner sep = 2.5pt] (A2) {};
			\draw (3.0+\shift,-0.4) node[star, fill, inner sep = 2pt] {};
			\draw (3.75+\shift,-0.4) node[star, fill, inner sep = 2pt] (A3) {};
			\draw (4.5+\shift,-0.4) node[star, fill, inner sep = 2pt] {};
			\draw (0.75,-0.4) node[diamond, fill, inner sep = 2pt] (A4) {};
			\draw (1.5,-0.4) node[diamond, fill, inner sep = 2pt] {};
			\draw (4.5+\shift+0.4,0) node[rectangle, fill, inner sep = 2pt] {};
			\draw (4.5+\shift+0.4,0.75) node[rectangle, fill, inner sep = 2pt] (A5) {};
			\draw (4.5+\shift+0.4,1.4) node[rectangle, fill, inner sep = 2pt] (A6) {};
			
			\draw (1.5,-1.5-\shift) node[diamond, fill, inner sep = 2pt] {};
			\draw (1.5,-2.2-\shift) node[diamond, fill, inner sep = 2pt] (B1) {};
			\draw (1.5,-3.0-\shift) node[diamond, fill, inner sep = 2pt] (B2) {};
			\draw (0.75,-1.5-\shift) node[diamond, fill, inner sep = 2pt] (B3) {};
			\draw (0.75,-2.2-\shift) node[diamond, fill, inner sep = 2pt] {};
			\draw (0.75,-3.0-\shift) node[diamond, fill, inner sep = 2pt] {};
			\draw (1.9,-1.5-\shift) node[star, fill, inner sep = 2pt] {};
			\draw (1.9,-2.0-\shift) node[star, fill, inner sep = 2pt] (B4){};
			\draw (1.9,-2.75-\shift) node[star, fill, inner sep = 2pt] (B5) {};
			\draw (0.25,-1.1-\shift) node[circle, fill, inner sep = 2.5pt] (B6) {};
			\draw (1.85,-1.1-\shift) node[circle, fill, inner sep = 2.5pt] {};
			\draw (1.85,-0.5-\shift) node[circle, fill, inner sep = 2.5pt] {};
			
			\draw (3.0+\shift,-1.5-\shift) node[star, fill, inner sep = 2pt] {};
			\draw (3.0+\shift,-2.0-\shift) node[star, fill, inner sep = 2pt] (C1) {};
			\draw (3.0+\shift,-2.75-\shift) node[star, fill, inner sep = 2pt] (C2) {};
			\draw (3.75+\shift,-1.5-\shift) node[star, fill, inner sep = 2pt] (C3) {};
			\draw (3.75+\shift,-2.0-\shift) node[star, fill, inner sep = 2pt] {};
			\draw (3.75+\shift,-2.75-\shift) node[star, fill, inner sep = 2pt] {};
			\draw (4.5+\shift,-1.5-\shift) node[star, fill, inner sep = 2pt] {};
			\draw (4.5+\shift,-2.0-\shift) node[star, fill, inner sep = 2pt] (C4) {};
			\draw (4.5+\shift,-2.75-\shift) node[star, fill, inner sep = 2pt] (C5) {};
			\draw (2.6+\shift,-1.5-\shift) node[diamond, fill, inner sep = 2pt] {};
			\draw (2.6+\shift,-2.2-\shift) node[diamond, fill, inner sep = 2pt] (C6) {};
			\draw (2.6+\shift,-3.0-\shift) node[diamond, fill, inner sep = 2pt] (C7) {};
			\draw (3.35,-0.5-\shift) node[circle, fill, inner sep = 2.5pt]  {};
			\draw (3.35,-1.1-\shift) node[circle, fill, inner sep = 2.5pt] {};
			\draw (4.5+\shift,-1.1-\shift) node[circle, fill, inner sep = 2.5pt] {};
			\draw (4.5+\shift+0.4,-1.5-\shift) node[regular polygon,regular polygon sides=3, fill, inner sep = 1.5pt] {};
			\draw (4.5+\shift+0.4,-3-\shift) node[regular polygon,regular polygon sides=3, fill, inner sep = 1.5pt] (C8) {};
			
			\draw (6+\shift,0) node[rectangle, fill, inner sep = 2pt] {};
			\draw (6.5+\shift,0) node[rectangle, fill, inner sep = 2pt] (D1) {};
			\draw (6+\shift,0.75) node[rectangle, fill, inner sep = 2pt] (D2) {};
			\draw (6.5+\shift,0.75) node[rectangle, fill, inner sep = 2pt] {};
			\draw (6+\shift,1.4) node[rectangle, fill, inner sep = 2pt] (D3) {};
			\draw (6.5+\shift,1.4) node[rectangle, fill, inner sep = 2pt] {};
			\draw (7.3+\shift,0) node[rectangle, fill, inner sep = 2pt] (D4) {};
			\draw (7.3+\shift,0.75) node[rectangle, fill, inner sep = 2pt] {};
			\draw (7.3+\shift,1.4) node[rectangle, fill, inner sep = 2pt] {};
			\draw (6+\shift-0.4,0) node[circle, fill, inner sep = 2.5pt] {};
			\draw (6+\shift-0.4,1.) node[circle, fill, inner sep = 2.5pt] (D5) {};
			\draw (6+\shift,-0.4) node[regular polygon,regular polygon sides=3, fill, inner sep = 1.5pt] {};
			\draw (7.3+\shift,-0.4) node[regular polygon,regular polygon sides=3, fill, inner sep = 1.5pt] (D6) {};
			
			\draw (6+\shift,-1.5-\shift) node[regular polygon,regular polygon sides=3, fill, inner sep = 1.5pt] {};
			\draw (7.3+\shift,-1.5-\shift) node[regular polygon,regular polygon sides=3, fill, inner sep = 1.5pt] (E1) {};
			\draw (6+\shift,-3.-\shift) node[regular polygon,regular polygon sides=3, fill, inner sep = 1.5pt] (E2) {};
			\draw (7.3+\shift,-3.-\shift) node[regular polygon,regular polygon sides=3, fill, inner sep = 1.5pt] {};
			\draw (6+\shift-0.4,-1.5-\shift) node[star, fill, inner sep = 2pt] {};
			\draw (6+\shift-0.4,-2.-\shift) node[star, fill, inner sep = 2pt] (E3) {};
			\draw (6+\shift-0.4,-2.75-\shift) node[star, fill, inner sep = 2pt] (E4) {};
			\draw (6+\shift,-1.1-\shift) node[rectangle, fill, inner sep = 2pt] {};
			\draw (6.5+\shift,-1.1-\shift) node[rectangle, fill, inner sep = 2pt] (E5) {};
			\draw (7.3+\shift,-1.1-\shift) node[rectangle, fill, inner sep = 2pt] (E6) {};
			
			\draw (1.85,-1.1-\shift) -- (2.30,-1.1-\shift) {};
			\draw (2.30,-1.1-\shift) -- (2.90,-0.5-\shift) {};
			\draw (2.90,-0.5-\shift) -- (3.35,-0.5-\shift) {};
			
			\draw (1.85,-0.5-\shift) -- (2.30,-0.5-\shift) {};
			\draw (2.30,-0.5-\shift) -- (2.54,-0.74-\shift) {};
			\draw (2.66,-0.86-\shift) -- (2.90,-1.1-\shift) {};
			\draw (2.90,-1.1-\shift) -- (3.35,-1.1-\shift) {};

			\draw[<->, line width = 1pt, latex-latex, bend right]
			(A1) edge (B6) (A4) edge (B3) (A3) edge (C3) (A5) edge (D2) (A6) edge (D3) 
			(E1) edge (D6) (E3) edge (C4) (E4) edge (C5);
			
			\draw[<->, line width = 1pt, latex-latex, bend left]
			(B4) edge (C1) (C6) edge (B1) (C7) edge (B2) 
			(E2) edge (C8) (E5) edge (D1) (E6) edge (D4);
			
			\draw[<->, line width = 1pt, latex-latex]
			(B5) edge (C2) (A2) edge (D5);
		\end{tikzpicture}
		\captionof{figure}{Action of the Lagrange multipliers \label{fig:ommiting}}
	\end{center}
\end{figure}
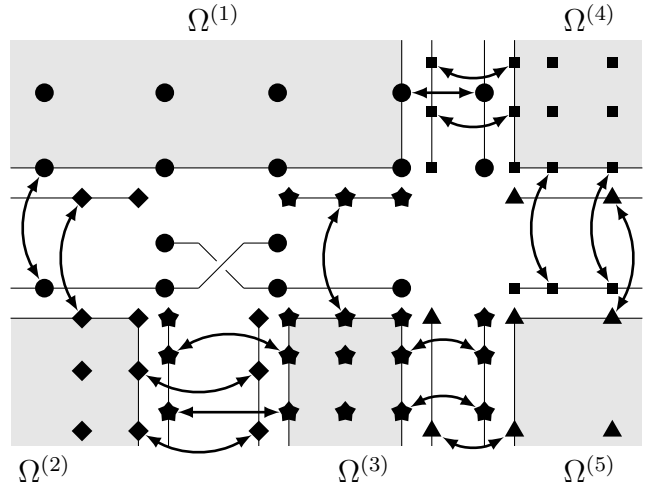

In the following, we introduce constraints that ensure the continuity
of the solution between the interface on one patch and the corresponding
artificial interfaces of the neighboring patches, i.e., between the
function in $V^{(k)}$ and the functions in $V^{(\ell,k)}$. Note that the
function spaces $V^{(k)}|_{\Gamma^{(k,\ell)}}$ and $V^{(\ell,k)}$
and the corresponding bases agree. This means that we obtain continuity
if the corresponding coefficients agree. Since the primal degrees of
freedom are enforced strongly, there are no constraints corresponding
to the primal degrees of freedom, see Figure~\ref{fig:ommiting}.

The constraints are represented by a matrix
\[
B =
\begin{pmatrix}
	B^{(1)} & \cdots & B^{(K)}
\end{pmatrix},
\]
where each row corresponds to one constraint enforcing the agreement
of one of the coefficients in the usual way, i.e., such that 
every row has two non-zero entries: $+1$ and $-1$. The choice
\[
B \underline v = 0,
\]
corresponds to a function $v$ that satisfies the constraints.
Note that the proposed choice of primal degrees of freedom guarantees
that there is no degree of freedom which is affected
by more than one constraint. This is a property, which we use in the
condition number analysis.

\subsection{IETI system}\label{subs:4:4}

Before we are able to setup the overall IETI system, we
partition the degrees of freedom into
the primal degrees of freedom (index $\Pi$),
the degrees of freedom, which are subject to Lagrange
multipliers, (index $\Delta$), and
the remaining, i.e., interior, degrees of freedom (index $\mathrm I$).
Using this partitioning, the matrices $A^{(k)}$, $B^{(k)}$
and the vector $\underline{f}_e^{(k)}$ have the following form:
\[
A^{(k)}
=
\begin{pmatrix}
	A_{\mathrm I \mathrm I}^{(k)} & A_{\mathrm I\Delta}^{(k)} & A_{\mathrm I\Pi}^{(k)}\\
	A_{\Delta \mathrm I}^{(k)} & A_{\Delta\Delta}^{(k)} & A_{\Delta\Pi}^{(k)} \\
	A_{\Pi \mathrm I}^{(k)} & A_{\Pi\Delta}^{(k)} & A_{\Pi\Pi}^{(k)} \\
\end{pmatrix}
,
\quad
\raisebox{1.2em}{
	$
	B^{(k)}
	=
	\underbrace{
		\begin{pmatrix}
			B_{\mathrm I}^{(k)} & B_{\Delta}^{(k)} & B_{\Pi}^{(k)}
		\end{pmatrix}
	}_{\displaystyle =
		\begin{pmatrix}
			0 & B_{\Delta}^{(k)} & 0
		\end{pmatrix}
	}
	$
}
,
\quad
\mbox{and}
\quad
\underline{f}_e^{(k)}
=
\begin{pmatrix}
	\underline{f}_{\mathrm I}^{(k)}\\
	\underline{f}_{\Delta}^{(k)}\\
	\underline{f}^{(k)}_{\Pi}
\end{pmatrix}
.
\]
Here, we make use of the fact that there are no Lagrange
multipliers that are acting on the interior degrees of freedom
(and thus $B_{\mathrm I}^{(k)}=0$) or on the primal degrees of
freedom (and thus $B_{\Pi}^{(k)}=0$).

On each patch we eliminate the primal degrees of freedom. So, we define
\[
\widetilde A^{(k)} :=
\begin{pmatrix}
	A_{\mathrm{I}\mathrm{I}}^{(k)} & A_{\mathrm{I}\Delta}^{(k)} \\
	A_{\Delta\mathrm{I}}^{(k)} & A_{\Delta\Delta}^{(k)} \\
\end{pmatrix}
,\quad
\raisebox{1.2em}{
	$\widetilde B^{(k)} :=
	\underbrace{
		\begin{pmatrix}
			B_{\mathrm{I}}^{(k)} & B_{\Delta}^{(k)}
		\end{pmatrix}
	}_{\displaystyle 		=
		\begin{pmatrix}
			0 & B_{\Delta}^{(k)}
		\end{pmatrix}
	}
	$},
\quad\mbox{and}\quad
\widetilde{\underline{f}}^{(k)} :=
\begin{pmatrix}
	\underline{f}_{\mathrm I}^{(k)}\\
	\underline{f}_{\Delta}^{(k)}\\
\end{pmatrix}
.
\]
These local matrices are collected to global matrices $A = \text{diag}(A^{(1)}, \dots, A^{(K)})$,
$\widetilde A = \text{diag}(\widetilde A^{(1)}, \dots, \widetilde A^{(K)})$ and
$\widetilde B = (\widetilde B^{(1)}, \dots, \widetilde B^{(K)})$,
and the local vectors into global vectors
$\widetilde{\underline{f}} = ((\widetilde{\underline{f}}^{(1)})^\top, \dots, (\widetilde{\underline{f}}^{(K)})^\top)^\top$
and
${\underline{f}} = (({\underline{f}_e}^{(1)})^\top, \dots, ({\underline{f}_e}^{(K)})^\top)^\top$.

For the setup of the primal problem, we introduce an $A$-orthogonal basis.
To do so, we first introduce for each patch an
$A^{(k)}$-orthogonal basis via
\[
\Psi^{(k)}
:=
\begin{pmatrix}
	A_{\mathrm{I}\mathrm{I}}^{(k)} & A_{\mathrm{I}\Delta}^{(k)}&0\\
	A_{\Delta\mathrm{I}}^{(k)} & A_{\Delta\Delta}^{(k)}&0\\		
	0 & 0 & I \\
\end{pmatrix}^{-1}
\begin{pmatrix}
	- A_{\mathrm{I}\Pi}^{(k)}\\
	- A_{\Delta\Pi}^{(k)}		\\
	I
\end{pmatrix}.
\]
Let $R^{(k)}$ be a binary matrix that restricts a global coefficient
vector of the primal degrees of freedom to the primal degrees that are
associated to space $V_e^{(k)}$. Then, we obtain the matrix representing
the global $A$-orthogonal basis for the primal degrees of freedom via
\[
\Psi := \begin{pmatrix}
	\Psi^{(1)}R^{(1)} \\ \vdots\\  \Psi^{(K)}R^{(K)}
\end{pmatrix} .
\]

The overall IETI-DP system reads as follows. Find
$(\widetilde{\underline{u}}^\top, \underline u_\Pi^\top, \underline \lambda^\top)^\top$ such that
\begin{equation}\label{eq:IETISystem}
	\begin{pmatrix}
		\widetilde{A}    &                   & \widetilde{B}^\top  \\
		& \Psi^\top A \Psi  & \Psi^\top B^\top    \\
		\widetilde{B}    & B \Psi            &                \\
	\end{pmatrix}
	\begin{pmatrix}
		\widetilde{\underline{u}}  \\
		\underline{u}_\Pi  \\
		\underline{\lambda}  \\
	\end{pmatrix}
	=
	\begin{pmatrix}
		\widetilde{\underline{f}}  \\
		\Psi^\top \underline{f}  \\
		0  \\
	\end{pmatrix}
	.
\end{equation}

This problem is equivalent to the original
problem~\eqref{discreteVarProb}, cf. \cite{MandelDohrmannTezaur:2005a}.
\begin{remark}
	In this paper, we follow the approach to eliminate the primal degrees
	of freedom, which is a commonly used approach for handling the
	corner values in actual implementations. 
	Alternatively, one can incorporate the primal constraints using Lagrange
	multipliers, which is a common approach if edge averages
	are used as primal degrees of freedom. Certainly, this approach is
	also possible in the framework of this paper. Here, we would obtain the
	formulation~\eqref{eq:IETISystem}, however with the choice
	\[
	\widetilde{A}^{(k)}
	=
	\begin{pmatrix}
		A_{\mathrm I \mathrm I}^{(k)} & A_{\mathrm I\Delta}^{(k)} & A_{\mathrm I\Pi}^{(k)} & 0\\
		A_{\Delta \mathrm I}^{(k)} & A_{\Delta\Delta}^{(k)} & A_{\Delta\Pi}^{(k)} & 0 \\
		A_{\Pi \mathrm I}^{(k)} & A_{\Pi\Delta}^{(k)} & A_{\Pi\Pi}^{(k)} & I \\
		0 & 0 & I & 0
	\end{pmatrix}
	,\quad
	\raisebox{1.2em}{$
		\widetilde B^{(k)}
		=
		\underbrace{
			\begin{pmatrix}
				B_{\mathrm I}^{(k)} & B_{\Delta}^{(k)} & B_{\Pi}^{(k)} & 0
		\end{pmatrix}}_{\displaystyle =
			\begin{pmatrix}
				0 & B_{\Delta}^{(k)} & 0 & 0
			\end{pmatrix}
		},
		$},
	\quad
	\widetilde{\underline f}_e^{(k)}
	=
	\begin{pmatrix}
		\underline f_{\mathrm I}^{(k)} \\
		\underline f_{\Delta}^{(k)} \\
		\underline f_{\Pi}^{(k)} \\
		0 
	\end{pmatrix}.
	\]
	The matrix $(0\;0\;I)$ in the definition of $\widetilde{A}^{(k)}$
	here is often called $C^{(k)}$.
	
	In theory papers, cf. \cite{MandelDohrmannTezaur:2005a}, a
	FETI-DP or IETI-DP system is often written down in the equivalent
	skeleton formulation, which corresponds to the choice
	\[
	\widetilde{A}^{(k)}
	=
	\begin{pmatrix}
		A_{\Delta\Delta}^{(k)} - A_{\Delta \mathrm I}^{(k)} (A_{\mathrm I \mathrm I}^{(k)})^{-1} A_{\mathrm I \Delta}^{(k)}
		& A_{\Delta\Pi}^{(k)} - A_{\Delta \mathrm I}^{(k)} (A_{\mathrm I \mathrm I}^{(k)})^{-1} A_{\mathrm I \Pi}^{(k)}
		& 0
		\\
		A_{\Pi\Delta}^{(k)} - A_{\Pi \mathrm I}^{(k)} (A_{\mathrm I \mathrm I}^{(k)})^{-1} A_{\mathrm I \Delta}^{(k)}
		& A_{\Pi\Pi}^{(k)} - A_{\Pi \mathrm I}^{(k)} (A_{\mathrm I \mathrm I}^{(k)})^{-1} A_{\mathrm I \Pi}^{(k)}
		& I \\
		0 &  I & 0
	\end{pmatrix}
	,
	\]
	\[
	\widetilde B^{(k)}
	=
	\begin{pmatrix}
		B_{\Delta}^{(k)} & B_{\Pi}^{(k)} & 0
	\end{pmatrix}
	=
	\begin{pmatrix}
		B_{\Delta}^{(k)} & 0 & 0
	\end{pmatrix},
	\quad\mbox{and}\quad
	\widetilde{\underline f}_e^{(k)}
	=
	\begin{pmatrix}
		\underline f_{\Delta}^{(k)} - A_{\Delta \mathrm I}^{(k)}(A_{\mathrm I \mathrm I}^{(k)})^{-1} \underline f_{\mathrm I}^{(k)} \\
		\underline f_{\Pi}^{(k)} - A_{\Pi \mathrm I}^{(k)}(A_{\mathrm I \mathrm I}^{(k)})^{-1} \underline f_{\mathrm I}^{(k)}\\
		0 
	\end{pmatrix}.
	\]
	
\end{remark}

\subsection{Solving the IETI system}\label{subs:4:5}

By applying a block-Gaussian elimination to~\eqref{eq:IETISystem},
we obtain the Schur complement equation 
\begin{equation}
	\label{IETIProblem}
	F \; \underline{\lambda} = \underline{d},
\end{equation}
for the Lagrange multipliers $\underline{\lambda}$,
where
\begin{align}
	\label{eq:IETI-matrix}
	F :=  \underbrace{
		\begin{pmatrix}
			\widetilde B &  B\Psi
		\end{pmatrix}
		\begin{pmatrix}
			\widetilde A \\ & \Psi^\top A \Psi
		\end{pmatrix}^{-1}
	}_{\displaystyle F_0:=}
	\begin{pmatrix}
		\widetilde B^\top  \\ \Psi^\top B^\top 
	\end{pmatrix}
	\quad\mbox{and}\quad
	\underline{d} :=  F_0\,
	\begin{pmatrix}
		\widetilde{\underline{f}}  \\
		\Psi^\top \underline{f}  \\
	\end{pmatrix}.
\end{align}
We solve~\eqref{IETIProblem} with a preconditioned conjugate gradient (PCG) solver. Let 
us define 
$B_\Gamma := 
\begin{pmatrix} B_{\Delta}^{(k)} & B_\Pi^{(k)} \end{pmatrix} = \begin{pmatrix} B_{\Delta}^{(k)} & 0 \end{pmatrix}$ and $B_\Gamma = (B_\Gamma^{(1)},\dots,B_\Gamma^{(K)})$. The preconditioner for the PCG method is the scaled Dirichlet preconditioner $M_{\mathrm{sD}}$ defined by
\[
M_{\mathrm{sD}} := B_\Gamma D^{-1} S  D^{-1} B_\Gamma^\top,
\]
where $S = \text{diag}(S^{(1)}, \dots, S^{(K)})$ with
\[
S^{(k)} :=
\begin{pmatrix}
	A_{\Delta\Delta}^{(k)} & A_{\Delta\Pi}^{(k)} \\
	A_{\Pi\Delta}^{(k)} & A_{\Pi\Pi}^{(k)} \\
\end{pmatrix}
-
\begin{pmatrix}
	A_{\Delta \mathrm I}^{(k)} \\
	A_{\Pi \mathrm I}^{(k)}  \\
\end{pmatrix}
(A_{\mathrm I \mathrm I}^{(k)})^{-1}
\begin{pmatrix}
	A_{\mathrm I\Delta}^{(k)} & A_{\mathrm I\Pi}^{(k)}\\
\end{pmatrix}
\]
is the restriction of the overall operator $A$ to the skeleton and
$D = \text{diag}(D^{(1)}, \dots, D^{(K)})$ is a diagonal matrix
defined based on the principle of coefficient scaling: Each
coefficient $d^{(k)}_{i,i}$ of $D^{(k)}$ is assigned  
\[
d^{(k)}_{i,i} := \frac{\alpha_k + \alpha_\ell}{\alpha_\ell}
\]
for degree of freedom $i$ associated to the interface $\Gamma^{(k,\ell)}$. If
$i$ corresponds to a primal degree of freedom, then $\ell$ can be chosen
arbitrarily among the indices of the neighboring patches.

After solving the system~\eqref{IETIProblem}, the solution vectors
$\widetilde{\underline{u}}$ and $ \underline u_\Pi$ and, finally,
$\underline u$ are computed from $\underline{u}$ by means of
simple patch-local postprocessing steps.

The execution of the IETI-DP method for the dG discretization described above requires basically the same computational steps as the IETI-DP method for dG discretizations on conforming patch decompositions, see~\cite[Section 3]{SchneckenleitnerTakacs:2020} for the detailed outline of the algorithm. 

\subsection{Condition number estimate}\label{subs:4:6}

The following theorem allows to estimate the maximum number
of iterations that the PCG solver with scaled Dirichlet preconditioner
needs to reach a desired error tolerance. The condition
number (and thus the number of iterations) depends on the patch sizes,
the grid size and on the spline degrees as explicitly stated in
the theorem (the constant $C$ does not depend on these quantities).
The dependence on the grid sizes and
patch sizes is as expected for FETI-like methods. Moreover, the
dependence on gird and patch sizes and spline degree is the same
as for the continuous case in IgA, see~\cite{SchneckenleitnerTakacs:2019}.
The condition number bound is independent of the
diffusion parameters $\alpha_k$, of the number of patches $K$, of the
continuity of the spline spaces, and of the choice of the penalty
parameter $\delta$ (provided $\delta$ is large enough such that the overall
bilinear form is coercive).
The constant $C$ (and thus the condition number) also depends on
the bounds for the geometry function, on the
number of patches that meet in a vertex and the quasi uniformity
of the grids.

\begin{theorem}\label{thrm:fin}
	Provided that the IETI-DP solver is set up as outlined in the previous
	sections, 
	\begin{itemize}
		\item 	
		there is a constant $C_1>0$ such that 
		\begin{align*}
			\sup_{x\in \overline{\widehat{\Omega}}}
			\| \nabla G_k(x) \|_{\ell^2} \le C_1\, H_{k}
			\quad\text{and}\quad
			\sup_{x\in \overline{\widehat{\Omega}}}
			\| (\nabla G_k(x))^{-1} \|_{\ell^2} \le C_1\, \frac{1}{H_{k}}
		\end{align*}
		for all $k=1,\ldots,K$, where $H_k:=\mathrm{diam}(\Omega^{(k)})$,
		\item there is a constant $C_2>0$ such that 
		\begin{equation}\label{eq:neighbors}
			| \{k \, : \, \mathbf{x} \in \partial \Omega^{(k)}, k = 1, \dots, K \} | 
			\leq C_2
		\end{equation}
		holds for all vertices $\mathbf{x}$,
		\item there is a constant $C_3>0$ such that
		\begin{equation}\label{eq:min:interface}
			C_3 H_k \leq |\Gamma^{(k,\ell)}|
		\end{equation}
		holds for all $k=1,\dots,K$ and all $\ell \in \mathcal{N}_\Gamma(k)$,
		\item and the grids are quasi-uniform, i.e.,
		there is a constant $C_4>0$ such that
		\[
		\widehat{h}_{k}
		\le C_4 \; \widehat{h}_{\mathrm{min},k}
		\]
		holds for all $k=1,\ldots,K$,
	\end{itemize}
	then the condition number of the preconditioned system satisfies
	\[
	\kappa(M_{\mathrm{sD}} F) \le C\, p \; 
	\left(1+\log p+\max_{k=1,\ldots,K} \log\frac{H_{k}}{h_{k}}\right)^2
	\]
	the constant $C$ only depends on the constants $C_1$, $C_2$, $C_3$
	and $C_4$.
\end{theorem}

In this section, we apply the IETI-DP method to a simple
magnetostatic problem. We consider the 
computational domain shown in Figure~\ref{fig:Cross section}
consisting of 272 patches representing a simplified cross section of an interior permanent magnet electric motor (IPMEM). The different colors in Fig.~\ref{fig:Cross section} denote different materials. The redbrown patches denote ferromagnetic material, e.g., iron, the yellow patches are the permanent magnets and the blue patches represent air regions and
coils made of copper (for which we use the same material parameters).

\begin{figure}[h]
	\centering
	\begin{subfigure}{.27\textwidth}
		\centering
		\includegraphics[scale=0.3]{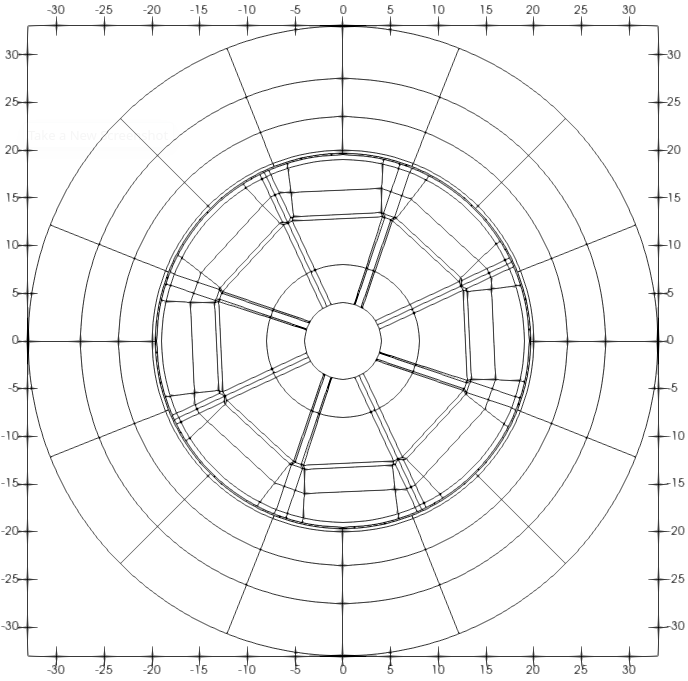}
	\end{subfigure} 
	\hspace{2.5cm}
	\begin{subfigure}{.27\textwidth}
		\centering
		\includegraphics[scale=0.3]{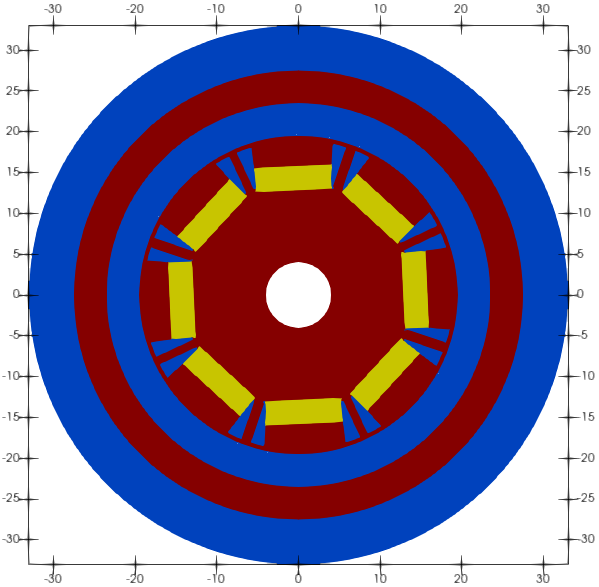}
	\end{subfigure} 
	\caption{Decomposition of cross section into patches (left),
		and their materials (right)}
	\label{fig:Cross section}
\end{figure}

The considered boundary reads formally as follows. Find $u$ such that
\[
\begin{aligned}
	- \mbox{div} (\nu(x,y) \nabla u(x,y)) & = \mbox{div}(\nu(x,y) M(x,y)) &&\qquad \mbox{for}\quad (x,y)\in\Omega \\
	u & = 0 &&\qquad \mbox{on}\quad \partial\Omega,
\end{aligned}
\]
where $\nu$ denotes the magnetic reluctivity and $M$ denotes the magnetization.
The magnetic reluctivity is $\nu_{\mathrm{ferro}}=\frac{1}{204 \pi}10^5$ on the ferromagnetic parts, $\nu_{\mathrm{mag}}=\frac{1}{4.344 \pi}10^7$ on the permanent magnets and $\nu_{\mathrm{air}}=\frac{1}{4 \pi}10^7$ on the air
and copper regions. This means that we have a jump in the order of approximately $10^4$. On each of the
permanent magnets, the magnetization $M$ is given by
\[
M = \rho_{\mathrm{mag}}\, \nu_{\mathrm{mag}}\, \textbf{n},
\]
where $\rho_{\mathrm{mag}}:=1.28$ is the 	magnetic remanence,
and $\textbf{n}$ is unit the normal vector in positive or negative
radial direction (measured from the center of the magnet), where
a positive sign is used for every second magnet and a negative sign
for every other second magnet. The magnetization $M$ vanishes
on the remainder of the domain (ferromagnetic parts, air, copper).

The cross section of the motor is modeled with NURBS and B-splines. For the coarsest discretization space, i.e., $r=0$, we use B-splines that are global polynomials and we use only splines of maximum smoothness within the patches. The subsequent refinements $r=1,2,3,\dots$ are obtained via uniform refinement steps. We solve the IETI-DP system~\eqref{IETIProblem} with the $M_\mathrm{sD}$ preconditioner that arises from the magnetostatic model problem with a PCG solver and start the iterations with zero initial vector. We stop the iteration if the $\ell_2$-norm of the residual has been decreased by a factor of $10^{-6}$ compared to the $\ell_2$-norm of the right-hand side. We use the penalty parameter $\delta = 12$ for all the numerical experiments which are carried out on the Radon1\footnote{https://www.ricam.oeaw.ac.at/hpc/} cluster located in Linz and we used the C++ library G+Smo~\cite{gismoweb}.

\begin{figure}[h]
	\centering
	\includegraphics[scale=0.3]{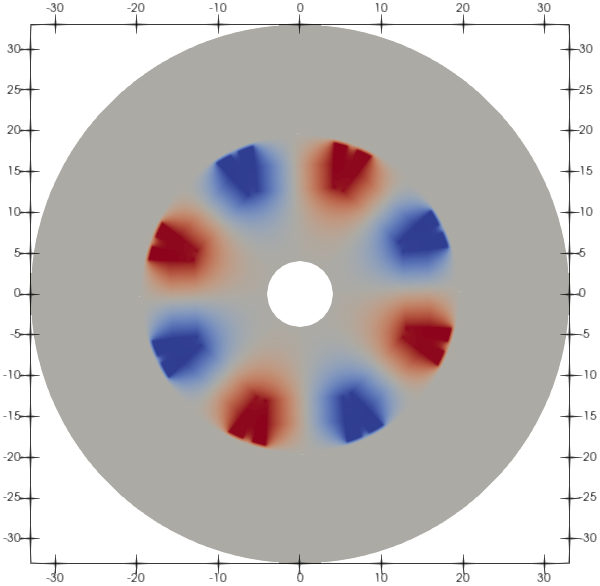}
	\caption{Solution of the linear model problem}
	\label{fig:Solution}
\end{figure}

Figure~\ref{fig:Solution} shows a typical solution to the problem.

Table~\ref{tab:IPMEM} shows the condition numbers and the iteration counts for the magnetostatic model problem. We see in this table a condition number growth with respect to $h$ as expected and we also see that the condition numbers decrease for higher polynomial degrees from refinement level $4$ on. 

\begin{table}[t]
	\newcolumntype{L}[1]{>{\raggedleft\arraybackslash\hspace{-1em}}m{#1}}
	\centering
	\renewcommand{\arraystretch}{1.25}
	\scalebox{1.}{
		\begin{tabular}{l|L{1em}L{2.em}|L{1em}L{2.em}|L{1em}L{2.em}|L{1em}L{2.em}|L{1em}L{2.em}|L{1em}L{2.em}}
			\toprule
			& \multicolumn{2}{c|}{$p=2$}
			& \multicolumn{2}{c|}{$p=3$}
			& \multicolumn{2}{c|}{$p=4$}
			& \multicolumn{2}{c|}{$p=5$}
			& \multicolumn{2}{c|}{$p=6$}
			& \multicolumn{2}{c}{$p=7$} \\
			$r$
			& it & $\kappa$
			& it & $\kappa$
			& it & $\kappa$
			& it & $\kappa$
			& it & $\kappa$
			& it & $\kappa$ \\
			\midrule
			1 & 26 &20.57 &	27 &20.46 &	27 &21.48 &	29 &24.07 &	30 &24.90 &	31 &27.00 \\
			2 & 28 &20.78 &	28 &23.13 &	30 &25.07 &	30 &26.89 &	31 &28.02 &	33 &29.69 \\
			3 & 32 &24.69 &	33 &27.46 &	34 &29.43 &	33 &30.83 &	32 &31.78 &	34 &33.23 \\
			4 & 43 &71.05 &	43 &64.84 &	44 &57.24 &	42 &49.15 &	41 &42.82 &	41 &39.25 \\
			5 & 47 &86.98 &	48 &87.11 &	47 &83.64 &	48 &79.40 &	47 &76.21 &	47 &70.17 \\
			6 & 50 &97.84 &	51 &99.09 &	50 &94.94 &	50 &96.83 &	52 &93.96 &	51 &91.89 \\
			\bottomrule
	\end{tabular}}
	\captionof{table}{Iterations (it) and condition numbers ($\kappa$); rotation angle of $\frac{5}{36}\pi$
		\label{tab:IPMEM}}
\end{table}

Table~\ref{tab:IPM_coef} reports on the robustness of the IETI-DP solver with respect to coefficient jumps. For the numerical tests we assign to the ferromagnetic patches the hypothetical reluctivity $10^{\mbox{j}}$ for $r = 5$ refinement steps. We see in the table that the condition number is almost independent of the coefficient jumps. 

\begin{table}[t]
	\newcolumntype{L}[1]{>{\raggedleft\arraybackslash\hspace{-1em}}m{#1}}
	\centering
	\renewcommand{\arraystretch}{1.25}
	\scalebox{1.0}{
		\begin{tabular}{l|L{1em}L{2.em}|L{1em}L{2.em}|L{1em}L{2.em}|L{1em}L{2.em}|L{1em}L{2.em}|L{1em}L{2.em}}
			\toprule
			& \multicolumn{2}{c|}{$p=2$}
			& \multicolumn{2}{c|}{$p=3$}
			& \multicolumn{2}{c|}{$p=4$}
			& \multicolumn{2}{c|}{$p=5$}
			& \multicolumn{2}{c|}{$p=6$}
			& \multicolumn{2}{c}{$p=7$} \\
			$j$
			& it & $\kappa$
			& it & $\kappa$
			& it & $\kappa$
			& it & $\kappa$
			& it & $\kappa$
			& it & $\kappa$ \\
			\midrule
			0 & 51 &88.85 &	50 &87.86 &	50 &81.78 &	48 &79.22 &	48 &73.44 &	49 &71.84 \\
			1 & 48 &89.63 &	47 &83.90 &	47 &81.45 &	47 &78.33 &	48 &76.15 &	47 &69.27 \\
			2 & 47 &82.00 &	47 &86.48 &	47 &83.74 &	47 &79.03 &	47 &74.39 &	47 &70.85 \\
			3 & 48 &87.90 &	49 &84.83 &	48 &81.92 &	48 &79.21 &	47 &75.37 &	47 &70.57 \\
			4 & 48 &81.66 &	49 &82.12 &	49 &80.32 &	48 &77.07 &	48 &72.57 &	47 &67.99 \\
			\bottomrule
	\end{tabular}}
	\captionof{table}{Iterations (it) and condition numbers ($\kappa$); $\nu$-robustness
		\label{tab:IPM_coef}}
\end{table}

In Table~\ref{tab:angle dependence}, we see the dependence of the iteration and condition numbers with respect to the angle of rotation of the motor. We choose three different angle positions $\varphi$ at $\frac{5}{36}\pi$, $\frac{6}{36}\pi$ and $\frac{7}{36}\pi$. We observe from this table that the iterations and condition numbers decrease for a larger angle $\varphi$.

\begin{table}[t]
	\newcolumntype{L}[1]{>{\raggedleft\arraybackslash\hspace{-1em}}m{#1}}
	\centering
	\renewcommand{\arraystretch}{1.25}
	\begin{tabular}{l|L{1em}L{1.8em}|L{1em}L{1.8em}|L{1em}L{1.8em}|L{1em}L{1.8em}|L{1em}L{1.8em}|p{3.5em}}
		\toprule
		%
		& \multicolumn{2}{c|}{$p=3$}
		& \multicolumn{2}{c|}{$p=4$}
		& \multicolumn{2}{c|}{$p=5$}
		& \multicolumn{2}{c|}{$p=6$}
		& \multicolumn{2}{c|}{$p=7$}
		& \\
		$r$
		& it & $\kappa$
		& it & $\kappa$
		& it & $\kappa$
		& it & $\kappa$
		& it & $\kappa$
		& \multicolumn{1}{c}{$\varphi$}  \\
		\midrule
		$4$  &	43 &64.84 &	44 &57.24 &	42 &49.15 &	41 &42.82 &	41 &39.25  
		& \rdelim\}{3}{3em}[\; \normalfont $\frac{5}{36}\pi$]  \\
		$5$  &	48 &87.11 &	47 &83.64 &	48 &79.40 &	47 &76.21 &	47 &70.17 \\
		$6$  &	51 &99.09 &	50 &94.94 &	50 &96.83 &	52 &93.96 &	51 &91.89 \\
		$4$  &	37 &33.77 &	37 &34.14 &	37 &35.33 &	38 &36.48 &	38 &37.55
		& \rdelim\}{3}{3em}[\; \normalfont $\frac{6}{36}\pi$] \\
		$5$  &	43 &55.63 &	43 &52.20 &	42 &49.48 &	42 &46.78 &	42 &45.63 \\
		$6$  &	44 &61.54 &	44 &59.93 &	44 &59.64 &	46 &58.46 &	45 &56.42 \\
		$4$  &	36 &31.88 &	37 &33.66 &	37 &35.01 &	37 &36.23 &	38 &37.60 
		& \rdelim\}{3}{3em}[\; \normalfont $\frac{7}{36}\pi$]  \\
		$5$  &	39 &36.89 &	40 &38.35 &	39 &39.50 &	39 &40.56 &	40 & 41.68 \\
		$6$  &	41 &41.82 &	41 &42.96 &	41 &44.30 &	42 &45.45 &	42 & 46.55 \\
		\bottomrule
	\end{tabular}
	\captionof{table}{Iteration counts (it) and condition numbers ($\kappa$); rotation dependence
		\label{tab:angle dependence}}
\end{table}

The Table~\ref{tab:parallel performance} shows the solving times in seconds (sec.) required to solve the IETI-DP system with the number of computing cores given in the column $\mathrm{proc}$. We increase the number of cores from $2$ to $16$. We see in this table a very good scaling behavior of the algorithm. 

\begin{table}[t]
	\newcolumntype{L}[1]{>{\raggedleft\arraybackslash\hspace{-1em}}m{#1}}
	\centering
	\renewcommand{\arraystretch}{1.25}
	\begin{tabular}{r|r|L{2.8em}|L{2.8em}|L{2.8em}|L{2.8em}|L{2.8em}}
		\toprule
		\multicolumn{1}{c|}{}
		& \multicolumn{1}{c|}{}
		& \multicolumn{1}{c|}{$p=3$}
		& \multicolumn{1}{c|}{$p=4$}
		& \multicolumn{1}{c|}{$p=5$}
		& \multicolumn{1}{c|}{$p=6$}
		& \multicolumn{1}{c}{$p=7$} \\
		\multicolumn{1}{c|}{$\mathrm{proc}$}
		& \multicolumn{1}{c|}{$r$}
		& \multicolumn{1}{c|}{time}
		& \multicolumn{1}{c|}{time}
		& \multicolumn{1}{c|}{time}
		& \multicolumn{1}{c|}{time}
		& \multicolumn{1}{c}{time} \\
		\midrule
		2 &4 &	8.33 &	10.97 &	11.78 &	13.23 &	16.67 \\
		4 &4 &	4.26 &	5.62 &	6.09 &	6.85 &	8.57 \\
		8 &4 &	2.14 &	2.89 &	3.07 &	3.59 &	4.40 \\
		16 &4 &	1.09 &	1.46 &	1.58 &	1.87 &	2.23 \\
		2 &5 &	34.46 &	44.5 &	58.43 &	74.14 &	93.90 \\
		4 &5 &	17.57 &	23.08 &	29.84 &	38.27 &	48.52 \\
		8 &5 &	8.91 &	12.19 &	15.72 &	19.93 &	25.32 \\
		16 &5 &	4.57 &	6.29 &	7.90 &	10.08 &	12.86 \\
		2 &6 &	179.23 & 231.30 &	327.64 &	403.44 &	453.04 \\
		4 &6 &	91.35 &	119.14 & 165.17 &	203.06 &	232.80 \\
		8 &6 &	47.03 &	61.67 &	87.27 &	104.03 &	120.06 \\
		16 &6 &	23.93 &	31.27 &	45.38 &	53.73 &	61.79 \\
		2 &7 &	990.49 & 1356.09 & 1681.87 &	2692.77 &	OoM \\
		4 &7 &	499.52 & 700.10 & 873.77 &	1390.72 &	1567.99 \\
		8 &7 &	257.08 & 364.11 & 443.77 &	720.69 &	809.26 \\
		16 &7 &	129.83 & 190.08 & 227.36 &	366.35 &	413.41 \\
		\bottomrule
	\end{tabular}
	\captionof{table}{Time in seconds to solve~\eqref{IETIProblem} 
		\label{tab:parallel performance}}
\end{table}

\section{Conclusions}
\label{sec:6}
In this paper, we have constructed a IETI-DP algorithm for computational domains with a non-matching decomposition into patches. We have adapted the idea of using corner values as primal degrees of freedom (Alg.~A) from~\cite{SchneckenleitnerTakacs:2020} according to our requirements. In this paper, we have generalized this idea to T-junctions: We add basis functions that are supported on a vertex to the primal space. For this choice, we obtain the same $h$ and $p$-explicit condition number bounds as in~\cite{SchneckenleitnerTakacs:2020}.

\section*{Appendix}
\renewcommand{\thesection}{\Alph{section}}
\setcounter{section}{1}
\label{sec:7}
In the appendix, we give a proof of Theorem~\ref{thrm:fin}.
Throughout this appendix, 
we use the notation $a \lesssim b$ if there is a constant
$c>0$ that only depends on the constants $C_1$, $C_2$, $C_3$
and $C_4$ from Theorem~\ref{thrm:fin}
such that $a\le cb$. Moreover, we write $a\eqsim b$ if $a\lesssim b\lesssim a$.

When it is clear from the context, we do not denote the restriction of
a function to an interface explicitly, so we write for example
$\|u^{(k)}\|_{L_2(\Gamma^{(k,\ell)})}$
instead of
$\|u^{(k)}|_{\Gamma^{(k,\ell)}}\|_{L_2(\Gamma^{(k,\ell)})}$.

Following the usual approach, for the analysis, we need to introduce
the skeleton representation of the solution which is obtained by eliminating the interior degrees of freedom. By eliminating the
interior degrees of freedom from the spaces $V^{(k)}$, we obtain the
space $W^{(k)}:=\{ v|_{\partial \Omega^{(k)}} \,:\, v \in V^{(k)}\}$.
The introduction of the skeleton representation has no influence on
the function spaces on the artificial interfaces, thus we define
$W^{(k,\ell)}:=V^{(k,\ell)}$. Based on these choices, we define
analogously to $V$ and $V_e^{(k)}$ the function spaces
\[
W := W_e^{(1)} \times \cdots \times W_e^{(K)},
\quad\mbox{and}\quad
W_e^{(k)} := W^{(k)}
\times \prod_{\ell \in \mathcal{N}_\Gamma(k)} W^{(k,\ell)}.
\]
Analogously to \eqref{def:representation}, a function $w_e^{(k)}\in W_e^{(k)}$ has the form $w_e^{(k)} = \left(w^{(k)}, \left(w^{(k,\ell)}\right)_{\ell\in \mathcal{N}_\Gamma(k)} \right)$, where
$w^{(k)} \in W^{(k)}$ and $w^{(k,\ell)}\in W^{(k,\ell)}$.
A basis for $W_e^{(k)}$ is canonically defined by choosing the
traces of the basis functions of the basis of $V^{(k)}$ (for which
the trace does not vanish) and the basis functions of the bases
of $V^{(k,\ell)} = W^{(k,\ell)}$.
Finally $\widetilde{W}\subseteq W$ is the subspace of functions
where the primal constraints are satisfied, i.e., the coefficients
for the vertex basis functions agree.

As in \cite{SchneckenleitnerTakacs:2019}, we define the seminorm
\[
|v|_{L_\infty^0(T)}:= \inf_{c\in \mathbb R} \| v - c\|_{L_\infty(T)}
\]
for a continuous function $v$ over the set $T \subset \mathbb{R}^2$.
Moreover, we use the standard seminorm
\[
|v|_{H^{1/2}(T)}:= \int_T \int_T
\frac{(v(x)-v(y))^2}{\|x-y\|_{\ell^2}^2}
\,\mathrm{d}y
\,\mathrm{d}x
\]
for $T$ being the boundary or an edge of a patch.

The following lemma allows to estimate the action
of the matrix $B^\top_D B_\Gamma$, where we define $B_D:= B_\Gamma D^{-1}$.

\begin{lemma}\label{lem:bbt}
	Let $u=(u_e^{(1)},\cdots,u_e^{(K)})=((u^{(1)},(u^{(1,\ell)})_{\ell\in \mathcal N_\Gamma(1)}),\ldots)\in \widetilde W$ with coefficient vector $\underline u$
	and let $w=(w_e^{(1)},\cdots,w_e^{(K)})=((w^{(1)},(w^{(1,\ell)})_{\ell\in \mathcal N_\Gamma(1)}),\ldots)\in \widetilde W$ with coefficient vector $\underline w$
	be such that $\underline w = B_D^\top B_\Gamma \underline u$.
	Then, we have for each patch $\Omega^{(k)}$ and each interface $\Gamma^{(k,\ell)}$ that 
	\[
	w^{(k)}|_{\Gamma^{(k,\ell)}} = \frac{\alpha_\ell}{\alpha_k + \alpha_\ell} 
	(u^{(k)}|_{\Gamma^{(k,\ell)}} - u^{(\ell,k)}), \;\;
	w^{(k,\ell)} = \frac{\alpha_\ell}{\alpha_k + \alpha_\ell} 
	(u^{(k,\ell)} - u^{(\ell)}|_{\Gamma^{(k,\ell)})}).
	\]
\end{lemma}
\begin{proof}
	As in the proof of \cite[Lemma~4.3]{SchneckenleitnerTakacs:2020},
	we have
	\begin{equation}\label{eq:bbtsplit}
		\begin{aligned}
			&w^{(k)}|_{\Gamma^{(k,\ell)}} =
			\frac{\alpha_\ell}{\alpha_k + \alpha_\ell} 
			\\&\hspace{5em}
			\left( u^{(k)}|_{\Gamma^{(k,\ell)}} - u^{(\ell,k)}
			- 
			\sum_{T \in \Gamma^{(k,\ell)}} \sum_{(i,j)\in \mathcal{B}_T(k,\ell)} \left(
			u_i^{(k)}\varphi_i^{(k)}|_{\Gamma^{(k,\ell)}} - u_j^{(\ell,k)} \varphi_j^{(\ell,k)} \right)
			\right) 
			,
			\\
			&w^{(k,\ell)}|_{\Gamma^{(k,\ell)}} =
			\frac{\alpha_\ell}{\alpha_k + \alpha_\ell} 
			\\&\hspace{5em}
			\left( u^{(k,\ell)} - u^{(\ell)}|_{\Gamma^{(k,\ell)}}
			-  
			\sum_{T \in \Gamma^{(k,\ell)}} \sum_{(i,j)\in \mathcal{B}_T(\ell,k)} \left(
			u_i^{(k,\ell)}\varphi_i^{(k,\ell)} - u_j^{(\ell)} \varphi_j^{(\ell)}|_{\Gamma^{(k,\ell)}} \right)
			\right)
			,
		\end{aligned}
	\end{equation}
	where $\varphi_i^{(k)}$ denotes a basis function of the basis
	of $W^{(k)}$ and $\varphi_i^{(k,\ell)}$ denotes a basis function
	of the basis of $W^{(k,\ell)}$. The set $\mathcal{B}_T(k,\ell)$
	contains the pairs of indices of those basis functions in the bases
	for $W^{(k)}$ and $W^{(\ell,k)}$ which are subject to a primal
	constraint.
	
	We obtain the representation~\eqref{eq:bbtsplit} since the
	coefficients corresponding to all basis functions on the common edge
	$\Gamma^{(k,\ell)}$ are equal to $\pm\alpha_\ell/(\alpha_k+\alpha_\ell)$,
	except to the basis functions that correspond to the primal degrees
	of freedom. For the latter, the corresponding coefficients are $0$
	since the primal degrees of freedoms are not subject to the jump matrix.
	Thus, we subtract the latter.
	
	Note that $u\in \widetilde W$, which means that it satisfies the
	primal constraints. Hence the sum over the indices in
	$\mathcal{B}_T(k,\ell)$ 
	vanishes. Therefore, we immediately obtain the desired result.
\end{proof}

Lemma~\ref{lem:bbt} and the triangle inequality immediately yield
\begin{equation}\label{eq:bbt}
	\begin{aligned}
		\|w^{(k)} - w^{(k,\ell)} \|_{L_{2}(\Gamma^{(k,\ell)})}^2 
		& \lesssim 
		\frac{\alpha_\ell^2}{(\alpha_k + \alpha_\ell)^2} \left( \|u^{(k)} - u^{(k,\ell)}\|^2_{L_{2}(\Gamma^{(k,\ell)})}
		+ \|u^{(\ell)}- u^{(\ell,k)} \|^2_{L_{2}(\Gamma^{(k,\ell)})} \right),\\
		|w^{(k)}|_{H^{1/2}(\Gamma^{(k,\ell)})}^2  &\lesssim
		\frac{\alpha_\ell^2}{(\alpha_k + \alpha_\ell)^2} \left( |u^{(k)}|_{H^{1/2}(\Gamma^{(k,\ell)})}^2 +
		|u^{(\ell,k)}|_{H^{1/2}(\Gamma^{(k,\ell)})}^2 \right), \\
		|w^{(k)}|_{L_\infty^0(\Gamma^{(k,\ell)})}^2
		& \lesssim
		\frac{\alpha_\ell^2}{(\alpha_k + \alpha_\ell)^2} \left( |u^{(k)}|_{L_\infty^0(\Gamma^{(k,\ell)})}^2 +
		|u^{(\ell,k)}|_{L_\infty^0(\Gamma^{(k,\ell)})}^2  \right).
	\end{aligned}
\end{equation}

Here and in what follows, we write
$\widehat h_{k\ell} := \text{min} \{\widehat h_{k}, \widehat h_{\ell} \}$
and
$h_{k\ell} := \text{min} \{h_{k}, h_{\ell} \}$.

We estimate contributions from the artificial interfaces in the $H^{1/2}$- and $L_\infty^0$-seminorms. We start with the $H^{1/2}$ estimate.
\begin{lemma}\label{lem:hhalf:on:artif:edge}
	Let $u \in \widetilde{W}$. Then, the estimate
	\[
	\begin{aligned}
		&
		|u^{(k,\ell)}|_{H^{1/2}{(\Gamma^{(k,\ell)})}}^2
		&  \lesssim 
		|u^{(k)}|_{H^{1/2}{(\partial \Omega^{(k)})}}^2
		+
		|u^{(\ell)}|_{H^{1/2}{(\partial \Omega^{(\ell)})}}^2
		+ \frac{p^2}{h_{k\ell}} \| u^{(k,\ell)} - u^{(k)} \|_{L_2(\Gamma^{(k,\ell)})}^2 
	\end{aligned}
	\]
	holds.
\end{lemma}
\begin{proof}
	The triangle inequality yields
	\begin{equation}\label{eq:dG:si:estimate*}
		\begin{aligned}
			|u^{(k,\ell)}|_{H^{1/2}{(\Gamma^{(k,\ell)})}}^2
			&\leq 
			2|u^{(k)}|_{H^{1/2}{(\Gamma^{(k,\ell)})}}^2
			+
			2|u^{(k)} - u^{(k,\ell)}|_{H^{1/2}{(\Gamma^{(k,\ell)})}}^2.
		\end{aligned}
	\end{equation}
	To estimate $|u^{(k)} - u^{(k,\ell)}|_{H^{1/2}{(\Gamma^{(k,\ell)})}}$, we use the equivalence of the norms on the physical and the parameter domain,~\cite[Lemma 1]{SchneckenleitnerTakacs:2020}, and interpolation, cf.~\cite[Theorem 5.2, eq. (3)]{AdamsFournier:2003}, to obtain 
	\begin{equation}\label{eq:dG:si:estimate1}
		\begin{aligned}
			&|u^{(k)} - u^{(k,\ell)}|^2_{H^{1/2}{(\Gamma^{(k,\ell)})}}
			\eqsim
			|\widehat{u}^{(k)} - \widehat{u}^{(k,\ell)}|^2_{H^{1/2}{(\widehat{\Gamma}^{(k,\ell)})}}\\
			& \qquad \lesssim 
			\|\widehat{u}^{(k)} - \widehat{u}^{(k,\ell)}\|_{L_{2}{(\widehat{\Gamma}^{(k,\ell)})}}
			\|\widehat{u}^{(k)} - \widehat{u}^{(k,\ell)}\|_{H^{1}{(\widehat{\Gamma}^{(k,\ell)})}}\\
			& \qquad \eqsim
			\|\widehat{u}^{(k)} - \widehat{u}^{(k,\ell)}\|^2_{L_{2}{(\widehat{\Gamma}^{(k,\ell)})}} + 
			\|\widehat{u}^{(k)} - \widehat{u}^{(k,\ell)}\|_{L_{2}{(\widehat{\Gamma}^{(k,\ell)})}}
			|\widehat{u}^{(k)} - \widehat{u}^{(k,\ell)}|_{H^{1}{(\widehat{\Gamma}^{(k,\ell)})}} \\
			& \qquad \leq
			\frac{p^2}{\widehat{h}_{k\ell}}
			\|\widehat{u}^{(k)} - \widehat{u}^{(k,\ell)}\|^2_{L_{2}{(\widehat{\Gamma}^{(k,\ell)})}} + 
			\|\widehat{u}^{(k)} - \widehat{u}^{(k,\ell)}\|_{L_{2}{(\widehat{\Gamma}^{(k,\ell)})}}
			|\widehat{u}^{(k)} - \widehat{u}^{(k,\ell)}|_{H^{1}{(\widehat{\Gamma}^{(k,\ell)})}}.
		\end{aligned}
	\end{equation}
	Next, we rotate the patches such that $\widehat{\Gamma}^{(k,\ell)}$ is $(a_1,a_2) \times \{0\}$. We define 
	the function $\widetilde{u}^{(k,\ell)}:=\widehat{u}^{(k,\ell)}(\cdot,0)$ and we denote by $\widetilde{\zeta}_1$ and $\widetilde{\zeta}_2$ the smallest and largest breakpoints, respectively, corresponding to the basis $\widehat{\Phi}^{(k,\ell)}$ such that $\widetilde{\zeta}_1 \geq a_1$ and $\widetilde{\zeta}_2 \leq a_2$.
	Using the triangle inequality, we estimate $|\widehat{u}^{(k)} - \widehat{u}^{(k,\ell)}|_{H^{1}{(\widehat{\Gamma}^{(k,\ell)})}}$ as 
	\[
	\begin{aligned}
		&|\widehat{u}^{(k)} - \widehat{u}^{(k,\ell)}|_{H^{1}{(\widehat{\Gamma}^{(k,\ell)})}} \\
		&\qquad \lesssim 	
		|\widehat{u}^{(k)} |_{H^1{(\widehat{\Gamma}^{(k,\ell)})}} 
		+ 
		|\widetilde{u}^{(k,\ell)} |_{H^{1}{((a_1,a_2)})} \\
		&\qquad \lesssim 
		|\widehat{u}^{(k)} |_{H^{1}{(\widehat{\Gamma}^{(k,\ell)})}}
		+
		|\widetilde{u}^{(k,\ell)} |_{H^{1}{((\widetilde{\zeta}_1,\widetilde{\zeta}_2)})}
		+ 
		|\widetilde{u}^{(k,\ell)} |_{H^{1}{((a_1,a_2)\backslash (\widetilde{\zeta}_1,\widetilde{\zeta}_2)})}.	
	\end{aligned}
	\]
	On $(\widetilde{\zeta}_1, \widetilde{\zeta}_2)$, we use~\cite[Proposition 2.2]{Heuer:2014} and apply an inverse inequality of~\cite[Lemma 4.3]{SchneckenleitnerTakacs:2019}. On $(a_1,a_2)\backslash (\widetilde{\zeta}_1,\widetilde{\zeta}_2)$, we use the fact that $u \in \widetilde{W}$. Hence, $\widetilde{u}^{(k,\ell)} = \widetilde{u}^{(\ell)}$ on $(a_1,a_2)\backslash (\widetilde{\zeta}_1,\widetilde{\zeta}_2)$ and we estimate 
	\begin{equation}\label{eq:dG:si:estimate:H1}
		\begin{aligned}
			&|\widehat{u}^{(k)} - \widehat{u}^{(k,\ell)}|_{H^{1}{(\widehat{\Gamma}^{(k,\ell)})}} \\
			& \quad \lesssim 
			|\widehat{u}^{(k)} |_{H^{1}{(\partial \widehat{\Omega})}}
			+
			\frac{p}{(\widehat{h}_{k\ell})^{1/2}}
			|\widehat{u}^{(k,\ell)} |_{H^{1/2}{(\widehat{\Gamma}^{(k,\ell)})})}
			+ 
			|\widehat{u}^{(\ell)} |_{H^{1}{(\partial \widehat{\Omega})}}	\\
			& \quad \lesssim 
			\frac{p}{(\widehat{h}_{k\ell})^{1/2}}
			\left(
			|\widehat{u}^{(k)} |_{H^{1/2}{(\partial \widehat{\Omega})}}
			+
			|\widehat{u}^{(k,\ell)} |_{H^{1/2}{(\widehat{\Gamma}^{(k,\ell)})})}
			+ 
			|\widehat{u}^{(\ell)} |_{H^{1/2}{(\partial \widehat{\Omega})}}	
			\right). 
		\end{aligned}
	\end{equation}
	We insert~\eqref{eq:dG:si:estimate:H1} into~\eqref{eq:dG:si:estimate1} and further into~\eqref{eq:dG:si:estimate*}, apply the norm equivalence between the parameter and physical domain,~\cite[Lemma 1]{SchneckenleitnerTakacs:2020}, 
	and denote by $c_1, c_2 > 0$ the hidden constants in the estimate to get
	\begin{equation*}
		\begin{aligned}
			|u^{(k,\ell)}|^2_{H^{1/2}{(\Gamma^{(k,\ell)})}}
			& \leq
			2 |u^{(k)}|_{H^{1/2}(\Gamma^{(k,\ell)})}^2 + 
			\frac{c_1 p^2}{h_{k\ell}}\|u^{(k)} - u^{(k,\ell)}\|^2_{L_{2}{(\Gamma^{(k,\ell)})}} 
			\\ & \qquad + 
			\frac{c_2 p}{(h_{k\ell})^{1/2}} 
			\|u^{(k)} - u^{(k,\ell)}\|_{L_{2}{(\Gamma^{(k,\ell)})}} \\
			& \qquad \qquad 
			\left(
			|{u}^{(k)} |_{H^{1/2}{(\partial \Omega^{(k)}})}
			+
			|{u}^{(k,\ell)} |_{H^{1/2}{(\Gamma^{(k,\ell)})}}
			+ |{u}^{(\ell)} |_{H^{1/2}{(\partial \Omega^{(\ell)})}}	
			\right).
		\end{aligned}
	\end{equation*}
	Using $ab \leq a^2 + b^2/4$ and $(a+b)^2 \leq 2a^2 + 2b^2$,
	we obtain
	\begin{equation*}
		\begin{aligned}
			|u^{(k,\ell)}|^2_{H^{1/2}{(\Gamma^{(k,\ell)})}}
			&\leq
			2 |u^{(k)}|_{H^{1/2}{(\Gamma^{(k,\ell)})}}^2 +
			\frac{(c_1+c_2^2) p^2}{h_{k\ell}}\|u^{(k)} - u^{(k,\ell)}\|^2_{L_{2}{(\Gamma^{(k,\ell)})}} \\
			& \qquad +
			\frac{1}{2} |{u}^{(k,\ell)} |^2_{H^{1/2}{(\Gamma^{(k,\ell)})}}
			+
			|{u}^{(k)} |^2_{H^{1/2}{(\partial \Omega^{(k)}})}
			+ |{u}^{(\ell)} |^2_{H^{1/2}{(\partial \Omega^{(\ell)})}}.	
		\end{aligned}
	\end{equation*}
	We subtract $\frac{1}{2}|u^{(k,\ell)}|^2_{H^{1/2}(\Gamma^{(k,\ell)})}$ from the equation to get statement of the lemma.
\end{proof}

\begin{lemma}\label{lem:linf:on:artif:edge}  
	Let $u \in \widetilde{W}$. Then, 
	\[
	\begin{aligned}
		&|u^{(k,\ell)}|_{L_\infty^0(\Gamma^{(k,\ell)})}^2  \\
		& \quad \lesssim 
		| u^{(k)} |_{L_\infty^0(\Gamma^{(k,\ell)})}^2 +
		| u^{(k)} |_{H^{1/2}(\partial \Omega^{(k)})}^2 + 
		| u^{(\ell)} |_{H^{1/2}(\partial \Omega^{(\ell)})}^2 +
		\frac{p^2}{h_{k\ell}}\| u^{(k)} - u^{(k,\ell)} \|_{L_2(\Gamma^{(k,\ell)})}^2
	\end{aligned}
	\]
	holds.
\end{lemma}
\begin{proof}
	Using the triangle inequality, we obtain 
	\begin{equation}\label{eq:linf:art:int1}
		| u^{(k,\ell)} |_{L_\infty^0(\Gamma^{(k,\ell)})}^2
		\lesssim 
		| u^{(k)} |_{L_\infty^0(\Gamma^{(k,\ell)})}^2 + 
		| u^{(k)} - u^{(k,\ell)} |_{L_\infty^0(\Gamma^{(k,\ell)})}^2.
	\end{equation}
	We apply~\cite[Lemma 5]{SchneckenleitnerTakacs:2020} and the norm equivalence,~\cite[Lemma 1]{SchneckenleitnerTakacs:2020}, to the difference $| u^{(k)} - u^{(k,\ell)} |_{L_\infty^0(\Gamma^{(k,\ell)})}^2$ to get
	\[
	\begin{aligned}
		& | u^{(k)} - u^{(k,\ell)} |_{L_\infty^0(\Gamma^{(k,\ell)})}^2
		=
		| \widehat{u}^{(k)} - \widehat{u}^{(k,\ell)} |_{L_\infty^0(\widehat{\Gamma}^{(k,\ell)})}^2\\
		& \qquad \lesssim 
		\| \widehat{u}^{(k)} - \widehat{u}^{(k,\ell)} \|_{L_2(\widehat{\Gamma}^{(k,\ell)})}
		\| \widehat{u}^{(k)} - \widehat{u}^{(k,\ell)} \|_{H^1(\widehat{\Gamma}^{(k,\ell)})} \\
		& \qquad \lesssim 
		\| \widehat{u}^{(k)} - \widehat{u}^{(k,\ell)} \|^2_{L_2(\widehat{\Gamma}^{(k,\ell)})}
		+
		\| \widehat{u}^{(k)} - \widehat{u}^{(k,\ell)} \|_{L_2(\widehat{\Gamma}^{(k,\ell)})} 
		| \widehat{u}^{(k)} - \widehat{u}^{(k,\ell)} |_{H^1(\widehat{\Gamma}^{(k,\ell)})}\\
		& \qquad \leq 
		\frac{p^2}{\widehat{h}_{k\ell}} 
		\| \widehat{u}^{(k)} - \widehat{u}^{(k,\ell)} \|^2_{L_2(\widehat{\Gamma}^{(k,\ell)})}
		+
		\| \widehat{u}^{(k)} - \widehat{u}^{(k,\ell)} \|_{L_2(\widehat{\Gamma}^{(k,\ell)})} 
		| \widehat{u}^{(k)} - \widehat{u}^{(k,\ell)} |_{H^1(\widehat{\Gamma}^{(k,\ell)})}.
	\end{aligned}
	\]
	We use~\eqref{eq:dG:si:estimate:H1} to obtain
	\[
	\begin{aligned}
		|\widehat{u}^{(k)} - \widehat{u}^{(k,\ell)}|_{H^{1}{(\widehat{\Gamma}^{(k,\ell)})}}
		&\lesssim
		\frac{p}{(\widehat{h}_{k\ell})^{1/2}}
		\left(
		|\widehat{u}^{(k)} |_{H^{1/2}{(\partial \widehat{\Omega}})}
		+
		|\widehat{u}^{(k,\ell)} |_{H^{1/2}{(\widehat{\Gamma}^{(k,\ell)})}}
		+ |\widehat{u}^{(\ell)} |_{H^{1/2}{(\partial \widehat{\Omega})}}	
		\right).
	\end{aligned}
	\]
	An application of the norm equivalence,~\cite[Lemma 1]{SchneckenleitnerTakacs:2020}, yields the estimate
	\[
	\begin{aligned}
		| u^{(k)} - u^{(k,\ell)} |_{L_\infty^0(\Gamma^{(k,\ell)})}^2
		&\lesssim
		\frac{p^2}{h_{k\ell}} 
		\| u^{(k)} - u^{(k,\ell)} \|^2_{L_2(\Gamma^{(k,\ell)})}+
		\frac{p}{(h_{k\ell})^{1/2}}
		\| u^{(k)} - u^{(k,\ell)} \|_{L_2(\Gamma^{(k,\ell)})} \\
		& \qquad 
		\left(
		|{u}^{(k)} |_{H^{1/2}{(\partial \Omega^{(k)}})} +
		|{u}^{(k,\ell)} |_{H^{1/2}{({\Gamma}^{(k,\ell)})}} + |{u}^{(\ell)} |_{H^{1/2}{(\partial \Omega^{(\ell)})}}
		\right).
	\end{aligned}
	\]
	Using $a(b+c) \lesssim a^2 + b^2 + c^2$ and $(a+b)^2 \lesssim a^2 + b^2$, yield
	\[
	\begin{aligned}
		| u^{(k)} - u^{(k,\ell)} |_{L_\infty^0(\Gamma^{(k,\ell)})}^2
		&\lesssim 
		\frac{p^2}{h_{k\ell}} 
		\| u^{(k)} - u^{(k,\ell)} \|^2_{L_2(\Gamma^{(k,\ell)})} 
		+
		|{u}^{(k)} |^2_{H^{1/2}{(\partial \Omega^{(k)}})}
		\\ & \qquad +
		|{u}^{(k,\ell)} |^2_{H^{1/2}{(\Gamma^{(k,\ell)})}}
		+ |{u}^{(\ell)} |^2_{H^{1/2}{(\partial \Omega^{(\ell)})}}.
	\end{aligned}
	\]
	Lemma~\ref{lem:hhalf:on:artif:edge} and \eqref{eq:linf:art:int1} finish the proof of this lemma.
\end{proof}

Before we give a proof of the main theorem, we estimate the sum of
the corresponding seminorms over all patches.
\begin{lemma}\label{lem:4:10}
	Let $u$ and $w$ be as in Lemma~\ref{lem:bbt} and
	assume that~\eqref{eq:neighbors} holds. Then, we have
	\[
	\begin{aligned}
		&\sum_{k=1}^K  \sum_{\ell\in \mathcal{N}_\Gamma(k)}\alpha_k \left(
		|w^{(k)}|_{H^{1/2}(\Gamma^{(k,\ell)})}^2
		+
		|w^{(k)}|_{L^0_\infty(\Gamma^{(k,\ell)})}^2		
		\right)
		\\
		&\qquad\lesssim
		\sum_{k=1}^K  \sum_{\ell\in \mathcal{N}_\Gamma(k)} \alpha_k
		\left( |u^{(k)}|_{H^{1/2}(\partial \Omega^{(k)})}^2 
		+|u^{(k)}|_{L^0_\infty(\Gamma^{(k,\ell)})}^2 
		+ 
		\frac{p^2}{h_{k\ell}} \| u^{(k)} - u^{(k,\ell)}\|^2_{L_2(\Gamma^{(k,\ell)})}
		\right).
	\end{aligned}
	\]
\end{lemma}
\begin{proof}
	Within this proof, all norms refer to $\Gamma^{(k,\ell)}=
	\Gamma^{(\ell,k)}$. \eqref{eq:bbt} and~\eqref{eq:neighbors} 
	yield
	\[
	\begin{aligned}
		&\mathcal{A}:=\sum_{k=1}^K  \sum_{\ell\in \mathcal{N}_\Gamma(k)} \alpha_k
		\left(
		|w^{(k)}|_{H^{1/2}(\Gamma^{(k,\ell)})}^2
		+ |w^{(k)}|^2_{L^0_{\infty}(\Gamma^{(k,\ell)})}
		\right)
		\\&\qquad\qquad \lesssim
		\sum_{k=1}^K \sum_{\ell\in \mathcal{N}_\Gamma(k)}
		\frac{\alpha_k \alpha_\ell^2}{(\alpha_k + \alpha_\ell)^2}
		\Big(
		|u^{(k)}|_{H^{1/2}(\Gamma^{(k,\ell)})}^2
		+
		|u^{(\ell,k)}|_{H^{1/2}(\Gamma^{(k,\ell)})}^2
		\\ & \qquad\qquad\qquad\qquad +
		|u^{(k)}|_{L^0_\infty(\Gamma^{(k,\ell)})}^2
		+
		|u^{(\ell,k)}|_{L^0_\infty(\Gamma^{(k,\ell)})}^2
		\Big).
	\end{aligned}
	\]
	Using Lemmas~\ref{lem:hhalf:on:artif:edge} and~\ref{lem:linf:on:artif:edge}, we obtain
	\[
	\begin{aligned}
		&\mathcal{A} \lesssim
		\sum_{k=1}^K \sum_{\ell\in \mathcal{N}_\Gamma(k)}
		\frac{\alpha_k \alpha_\ell^2}{(\alpha_k + \alpha_\ell)^2}
		\Big(
		|u^{(k)}|_{H^{1/2}(\partial \Omega^{(k)})}^2
		+
		|u^{(\ell)}|_{H^{1/2}(\partial \Omega^{(\ell)})}^2
		\\ & \qquad\qquad +
		|u^{(k)}|_{L^0_\infty(\Gamma^{(k,\ell)})}^2
		+
		|u^{(\ell)}|_{L^0_\infty(\Gamma^{(k,\ell)})}^2
		+\frac{p^2}{h_{k\ell}} \| u^{(\ell)} - u^{(\ell,k)}\|^2_{L_2(\Gamma^{(k,\ell)})}
		\Big).
	\end{aligned}
	\]	
	The estimate $\frac{\alpha_k \alpha_\ell^2}{(\alpha_k + \alpha_\ell)^2} \leq \min\{\alpha_k, \alpha_\ell\}$,
	and $\ell\in\mathcal{N}_{\Gamma}(k) \Leftrightarrow k\in\mathcal{N}_{\Gamma}(\ell)$ yield the desired estimate.
\end{proof}

\begin{proof}[Proof of Theorem~\ref{thrm:fin}]
	The idea of the proof is to use \cite[Theorem~22]{MandelDohrmannTezaur:2005a},
	which states that 
	\begin{equation}\label{eq:MandelDohrmannTezaur}
		\kappa(M_{\mathrm{sD}} \, F) \le \sup_{u \in  \widetilde W }
		\frac{ \| B_D^\top B_\Gamma \underline u \|_S^2 }{ \| \underline u \|_S^2 },
	\end{equation}
	where $\underline u$ is the coefficient vector associated
	to the function $u=(u_e^{(1)},\cdots,u_e^{(K)})
	=((u^{(1)},(u^{(1,\ell)})_{\ell\in\mathcal N_\Gamma(k)}),\cdots)$.
	So, let $u$ be arbitrary but fixed
	and let the function $w=(w_e^{(1)},\cdots,w_e^{(K)})
	=((w^{(1)},(w^{(1,\ell)})_{\ell\in\mathcal N_\Gamma(k)}),\cdots)$
	with coefficient vector $\underline w$ be such that
	$
	\underline w = B_D^\top B_\Gamma \underline u
	$.
	The Schur complement norm of the function $w$ is equivalent to the dG norm of its discrete harmonic extension, cf.~\cite{SchneckenleitnerTakacs:2020}. This means that 
	\begin{equation}\label{eq:thetwosums}
		\begin{aligned}
			\| B_D^\top B_\Gamma \underline u \|_S^2 = \|  \underline w \|_S^2
			\eqsim \sum_{k=1}^{K} \alpha_k | \mathcal{H}_h^{(k)} w^{(k)} |_{H^{1}(\Omega^{(k)})}^2
			+ \sum_{k=1}^{K} \sum_{\ell\in\mathcal{N}_\Gamma(k)}
			\alpha_k \frac{\delta p^2}{h_{kl}}
			\| w^{(k)} -w^{(k,\ell)} \|_{L_2(\Gamma^{(k,\ell)})}^2,
		\end{aligned}
	\end{equation}
	where $\mathcal{H}_h^{(k)}: W^{(k)} \rightarrow V^{(k)}$ denotes the discrete harmonic extension that minimizes the energy with respect to the bilinear form $a^{(k)}(\cdot, \cdot)$.
	First, we estimate the first sum in~\eqref{eq:thetwosums}.
	\cite[Theorem 4.2]{SchneckenleitnerTakacs:2019} yields
	\[
	\sum_{k=1}^{K} \alpha_k |\mathcal{H}_h^{(k)} w^{(k)} |_{H^1(\Omega^{(k)})}^2
	\lesssim
	p\sum_{k=1}^{K} \alpha_k | w^{(k)} |_{H^{1/2}(\partial \Omega^{(k)})}^2.
	\]
	Using \cite[Lemma 4.15]{SchneckenleitnerTakacs:2019} and an analogous estimate for a single edge
	(which depends on~\eqref{eq:min:interface}), we get
	\[
	\begin{aligned}
		\sum_{k=1}^{K} \alpha_k |\mathcal{H}_h^{(k)} w^{(k)} |_{H^1(\Omega^{(k)})}^2 
		&	\lesssim
		p\sum_{k=1}^{K}  \sum_{\ell \in \mathcal N_\Gamma(k)}\alpha_k \left( | w^{(k)} |_{H^{1/2}(\Gamma^{(k,\ell)})}^2 + 
		\Lambda
		|w^{(k)}|_{L_\infty^0(\Gamma^{(k,\ell)})}^2 
		\right),
	\end{aligned}	
	\]
	where $\Lambda := 1 + \log p + \max_{k = 1,\dots,K} \log {\frac{H_k}{h_k}}$.
	Using $\Lambda\ge1$ and Lemma~\ref{lem:4:10}, we obtain further
	\[
	\begin{aligned}
		&\sum_{k=1}^{K} \alpha_k |\mathcal{H}_h^{(k)} w^{(k)} |_{H^1(\Omega^{(k)})}^2 
		\lesssim
		p  \Lambda 
		\sum_{k=1}^K \sum_{\ell\in \mathcal{N}_\Gamma(k)}
		\alpha_k \left( |u^{(k)}|_{H^{1/2}(\partial \Omega^{(k)})}^2 
		+|u^{(k)}|_{L_\infty^0(\Gamma^{(k,\ell)})}^2
		\right)\\
		&\qquad+
		p  \Lambda
		\sum_{k=1}^K \sum_{\ell\in \mathcal{N}_\Gamma(k)}
		\alpha_k \frac{p^2}{h_{k\ell}}\| u^{(k,\ell)} - u^{(k)} \|_{L_2(\Gamma^{(k,\ell)})}^2.
	\end{aligned}
	\]
	Using \cite[Lemma~4.15 and Theorem~4.2]{SchneckenleitnerTakacs:2019}
	and $|\mathcal{N}_\Gamma(k)|\lesssim 1$, we further estimate
	\[
	\begin{aligned}
		\sum_{k=1}^{K} \alpha_k |\mathcal{H}_h^{(k)} w^{(k)} |_{H^1(\Omega^{(k)})}^2 
		& \lesssim p\Lambda
		\sum_{k=1}^K \alpha_k |\mathcal{H}_h^{(k)}u^{(k)}|_{H^{1}(\Omega^{(k)})}^2 
		+ p \Lambda 
		\sum_{k=1}^K \sum_{\ell\in \mathcal{N}_\Gamma(k)}
		\alpha_k |u^{(k)}|_{L_\infty^0(\Gamma^{(k,\ell)})}^2
		\\ & \hspace{-2cm} + p \Lambda
		\sum_{k=1}^K \sum_{\ell\in \mathcal{N}_\Gamma(k)} \alpha_k
		\frac{p^2}{h_{k\ell}}\| u^{(k,\ell)} - u^{(k)} \|_{L_2(\Gamma^{(k,\ell)})}^2 
		.
	\end{aligned}
	\]
	Using \cite[Lemma~4.14]{SchneckenleitnerTakacs:2019}, we get
	further
	\[
	\begin{aligned}
		& \sum_{k=1}^{K} \alpha_k |\mathcal{H}_h^{(k)} w^{(k)} |_{H^1(\Omega^{(k)})}^2
		\lesssim
		p\Lambda
		\sum_{k=1}^K \alpha_k |\mathcal{H}_h^{(k)}u^{(k)}|_{H^{1}(\Omega^{(k)})}^2  \\
		& \quad 
		+ p \Lambda^2 \sum_{k=1}^K \alpha_k \sum_{\ell\in \mathcal{N}_\Gamma(k)}
		\left(
		|\mathcal{H}_h^{(k)}u^{(k)}|_{H^{1}(\Omega^{(k)})}^2 + \inf_{c\in \mathbb R} 
		\|(\mathcal{H}_h^{(k)}u^{(k)}-c) \circ G_k\|_{H^1(\widehat{\Omega})}^2 
		\right)\\
		&\quad + p \Lambda 
		\sum_{k=1}^K \sum_{\ell\in \mathcal{N}_\Gamma(k)} \alpha_k
		\frac{p^2}{h_{k\ell}}\| u^{(k,\ell)} - u^{(k)} \|_{L_2(\Gamma^{(k,\ell)})}^2 
		.
	\end{aligned}
	\]
	The assumption~\eqref{eq:neighbors}, the Poincar{\'e} inequality and the norm equivalence between the parameter and the physical domain,~\cite[Lemma 1]{SchneckenleitnerTakacs:2020},
	yield the estimate 
	\begin{equation}\label{eq:AlgA first}
		\begin{aligned}
			& \sum_{k=1}^{K} \alpha_k |\mathcal{H}_h^{(k)} w^{(k)} |_{H^1(\Omega^{(k)})}^2 \\
			& \quad \lesssim
			p\Lambda^2 
			\sum_{k=1}^K
			\left(
			\alpha_k |\mathcal{H}_h^{(k)}u^{(k)}|_{H^{1}(\Omega^{(k)})}^2 
			+  \sum_{\ell\in \mathcal{N}_\Gamma(k)}
			\alpha_k \frac{p^2}{h_{k\ell}}\| u^{(k,\ell)} - u^{(k)} \|_{L_2(\Gamma^{(k,\ell)})}^2
			\right).
		\end{aligned}
	\end{equation}
	The estimate~\eqref{eq:bbt}, $\frac{\alpha_k\alpha_\ell^2}{(\alpha_k+\alpha_\ell)^2}\le\min\{\alpha_k,\alpha_\ell\}$ and $\ell \in \mathcal{N}_\Gamma(k) \Leftrightarrow k \in \mathcal{N}_\Gamma(\ell)$ immediately yield for the second sum in~\eqref{eq:thetwosums},
	\begin{equation}
		\begin{aligned}\label{eq:second sum}
			\sum_{k=1}^{K}\sum_{\ell \in \mathcal N_\Gamma(k)} \alpha_k \frac{\delta p^2}{h_{k\ell}} 
			\| w^{(k)} - w^{(k,\ell)} \|_{L_2(\Gamma^{(k,\ell)})}^{2} 
			\lesssim 
			\sum_{k=1}^{K}\sum_{\ell \in \mathcal N_\Gamma(k)}
			\alpha_k
			\frac{\delta p^2}{h_{k\ell}}\| u^{(k)} - u^{(k,\ell)} \|_{L_2(\Gamma^{(k,\ell)})}^{2}.
		\end{aligned}
	\end{equation}
	Using the fact that $\Lambda\ge 1$, \eqref{eq:AlgA first} and~\eqref{eq:second sum} to estimate~\eqref{eq:thetwosums}, we obtain
	\[
	\begin{aligned}
		\| B_D^\top B_\Gamma \underline u \|_S^2 &\lesssim
		p \Lambda^2 \sum_{k=1}^{K} \left(
		\alpha_k | \mathcal{H}_h^{(k)} u^{(k)} |_{H^1(\Omega^{(k)})}^2 + 
		\sum_{\ell \in \mathcal N_\Gamma(k)} \alpha_k \frac{\delta p^2}{h_{k\ell}}
		\| u^{(k)} - u^{(k,\ell)} \|_{L_2(\Gamma^{(k,\ell)})}^{2} 
		\right) \\
		&= p \Lambda^2 \| u \|_{d}^2 
		\eqsim p \Lambda^2 \| \underline{u} \|_{S}^2 
	\end{aligned}
	\]
	The combination of this estimate and~\eqref{eq:MandelDohrmannTezaur}
	finishes the proof.
\end{proof}

\mbox{}\\[-2em]

\textbf{Acknowledgments.}
The first author was supported by the Austrian Science Fund (FWF): S117 and 
W1214-04. Also, the second author has received support from the Austrian Science
Fund (FWF): P31048.


%
%
\bibliography{literature}

\begin{thebibliography}{10}

\bibitem{AdamsFournier:2003}
R.~Adams and J.~Fournier.
\newblock {\em Sobolev Spaces}.
\newblock Elsevier Science, 2003.

\bibitem{AlottoBertoni:2001}
P.~Alotto, A.~Bertoni, I.~Perugia, and D.~Schoetzau.
\newblock Discontinuous finite element methods for the simulation of rotating
  electrical machines.
\newblock {\em COMPEL}, 20:448 -- 462, 2001.

\bibitem{Arnold:1982}
D.~Arnold.
\newblock An interior penalty finite element method with discontinuous
  elements.
\newblock {\em SIAM J. Numer. Anal.}, 19(4):742 -- 760, 1982.

\bibitem{BuffaMaday:2001}
A.~Buffa, Y.~Maday, and F.~Rapetti.
\newblock A sliding mesh-mortar method for a two dimensional eddy currents
  model of electric engines.
\newblock {\em ESAIM: Math. Model. Numer. Anal.}, 35(2):191 -- 228, 2001.

\bibitem{CafieroLloberas:2016}
M.~Cafiero, O.~Lloberas-Valls, J.~Cante, and J.~Oliver.
\newblock The domain interface method: a general-purpose non-intrusive
  technique for non-conforming domain decomposition problems.
\newblock {\em Comput. Mech.}, 57(4):555 -- 581, 2016.

\bibitem{Cottrell:Hughes:Bazilevs}
J.~A. Cottrell, T.~J.~R. Hughes, and Y.~Bazilevs.
\newblock {\em Isogeometric Analysis -- Toward Integration of {CAD} and {FEA}}.
\newblock John Wiley \& Sons, 2009.

\bibitem{DavatRen:1985}
B.~{Davat}, Z.~{Ren}, and M.~{Lajoie-Mazenc}.
\newblock The movement in field modeling.
\newblock {\em IEEE Trans. Magn.}, 21(6):2296 -- 2298, 1985.

\bibitem{DryjaGalvis:2013}
M.~Dryja, J.~Galvis, and M.~Sarkis.
\newblock A {FETI-DP} preconditioner for a composite finite element and
  discontinuous {G}alerkin method.
\newblock {\em SIAM J. Numer. Anal.}, 51(1):400 -- 422, 2013.

\bibitem{Egger:Harutyunyan:Merkel:Schops:2020}
H.~Egger, M.~Harutyunyan, M.~Merkel, and S.~Schöps.
\newblock On the stability of harmonic mortar methods with application to
  electric machines, 2020.
\newblock \url{https://arxiv.org/pdf/2005.12020.pdf}.

\bibitem{FarhatLesoinneLeTallecPiersonRixen:2001a}
C.~Farhat, M.~Lesoinne, P.~L. Tallec, K.~Pierson, and D.~Rixen.
\newblock {FETI-DP}: A dual-primal unified {FETI} method {I}: A faster
  alternative to the two-level {FETI} method.
\newblock {\em Int. J. Numer. Methods Eng.}, 50:1523 -- 1544, 2001.

\bibitem{FarhatRoux:1991a}
C.~Farhat and F.-X. Roux.
\newblock A method of finite element tearing and interconnecting and its
  parallel solution algorithm.
\newblock {\em Int. J. Numer. Methods Eng.}, 32(6):1205 -- 1227, 1991.

\bibitem{Heuer:2014}
N.~Heuer.
\newblock On the equivalence of fractional-order {S}obolev semi-norms.
\newblock {\em J. Math. Anal. Appl.}, 417(2):505--518, 2014.

\bibitem{Hofer:2016a}
C.~Hofer.
\newblock Analysis of discontinuous {G}alerkin dual-primal isogeometric tearing
  and interconnecting methods.
\newblock {\em Math. Models Methods Appl. Sci.}, 28(1):131 -- 158, 2018.

\bibitem{HoferLanger:2016a}
C.~Hofer and U.~Langer.
\newblock Dual-primal isogeometric tearing and interconnecting solvers for
  multipatch continuous and discontinuous {G}alerkin {IgA} equations.
\newblock {\em PAMM}, 16(1):747 -- 748, 2016.

\bibitem{HoferLanger:2017c}
C.~Hofer and U.~Langer.
\newblock {Dual-primal isogeometric tearing and interconnecting solvers for
  multipatch dG-IgA equations}.
\newblock {\em Comput. Methods Appl. Mech. Eng.}, 316:2 -- 21, 2017.

\bibitem{HoferLanger:2019b}
C.~Hofer, U.~Langer, and I.~Toulopoulos.
\newblock Isogeometric analysis on non-matching segmentation: discontinuous
  {G}alerkin techniques and efficient solvers.
\newblock {\em J. Appl. Math. Comput.}, 61(1):297 -- 336, 2019.

\bibitem{HughesCottrellBazilevs:2005}
T.~J.~R. Hughes, J.~A. Cottrell, and Y.~Bazilevs.
\newblock Isogeometric analysis: {CAD}, finite elements, {NURBS}, exact
  geometry and mesh refinement.
\newblock {\em Comput. Methods Appl. Mech. Eng.}, 194(39-41):4135 -- 4195,
  2005.

\bibitem{KettunenKurz:2014}
L.~{Kettunen}, S.~{Kurz}, T.~{Tarhasaari}, V.~{Raisanen}, A.~{Stenvall}, and
  S.~{Suuriniemi}.
\newblock Modeling rotation in electrical machines.
\newblock {\em IEEE Trans. Magn.}, 50(4):1 -- 10, 2014.

\bibitem{KleissPechsteinJuttlerTomar:2012}
S.~Kleiss, C.~Pechstein, B.~J{\"u}ttler, and S.~Tomar.
\newblock {IETI-Isogeometric Tearing and Interconnecting}.
\newblock {\em Comput. Methods Appl. Mech. Eng.}, 247-248:201 -- 215, 2012.

\bibitem{LaiRodger:1992}
H.~C. {Lai}, D.~{Rodger}, and P.~J. {Leonard}.
\newblock {Coupling meshes in 3D problems involving movements}.
\newblock {\em IEEE Trans. Magn.}, 28(2):1732 -- 1734, 1992.

\bibitem{MandelDohrmannTezaur:2005a}
J.~Mandel, C.~R. Dohrmann, and R.~Tezaur.
\newblock An algebraic theory for primal and dual substructuring methods by
  constraints.
\newblock {\em Appl. Numer. Math.}, 54(2):167 -- 193, 2005.

\bibitem{gismoweb}
A.~Mantzaflaris, R.~Schneckenleitner, S.~Takacs, and others~(see website).
\newblock {G+Smo (Geometry plus Simulation modules)}.
\newblock \url{http://github.com/gismo}, 2020.

\bibitem{PerrinCoulomb:1995}
R.~{Perrin-Bit} and J.~L. {Coulomb}.
\newblock A three dimensional finite element mesh connection for problems
  involving movement.
\newblock {\em IEEE Trans. Magn.}, 31(3):1920 -- 1923, 1995.

\bibitem{PrestonReece:1988}
T.~W. {Preston}, A.~B.~J. {Reece}, and P.~S. {Sangha}.
\newblock Induction motor analysis by time-stepping techniques.
\newblock {\em IEEE Trans. Magn.}, 24(1):471 -- 474, 1988.

\bibitem{SchneckenleitnerTakacs:2019}
R.~Schneckenleitner and S.~Takacs.
\newblock Condition number bounds for {IETI-DP} methods that are explicit in
  $h$ and $p$.
\newblock {\em Math. Models Methods Appl. Sci.}, 30(11):2067 -- 2103, 2020.

\bibitem{SchneckenleitnerTakacs:2020}
R.~Schneckenleitner and S.~Takacs.
\newblock Convergence theory for {IETI-DP} solvers for discontinuous {G}alerkin
  {I}sogeometric {A}nalysis that is explicit in $h$ and $p$.
\newblock {\em Comput. Methods Appl. Math.}, 2021.
\newblock Online first.

\bibitem{SchneckenleitnerTakacs:2021}
R.~Schneckenleitner and S.~Takacs.
\newblock Towards a {IETI-DP} solver on non-matching multi-patch domains.
\newblock In {\em Domain Decomposition Methods in Science and Engineering
  XXVI}, 2021.
\newblock To appear. \url{https://arxiv.org/pdf/2103.02536.pdf}.

\bibitem{Takacs:2019b}
S.~Takacs.
\newblock {Discretization error estimates for discontinuous Galerkin
  Isogeometric Analysis}.
\newblock {\em Appl. Anal.}, 2021.
\newblock To appear. \url{https://arxiv.org/abs/1901.03263}.

\end{thebibliography}

\end{document}